\documentclass [12pt,a4paper,reqno]{amsart}
 \input amssymb.sty
\usepackage{mathabx}

\newcommand{\ds}[1]{\ {#1} \ }
\newcommand{\dss}[1]{\quad {#1} \quad }

\input xy
\xyoption{all}

\def\Udot{\cdot_U}
\def\Mdot{\cdot_M}
\def\Rdot{\cdot_R}
\def\Ndot{\cdot_N}
\def\iso{\tilde{\to}}
\def\into{\hookrightarrow}
\def\Uz{U^0}

\def\vrpzv{\vrp^0_v}

\def\Real{\mathbb R}
\def\Int{\mathbb Z}

\def\00{\{ 0\}}

\def\chU{\widecheck{U}}

\def\chvrp{\widecheck{\vrp}}

\def\opp{\operatorname{opp}}

\def\full{\operatorname{full}}
\def\htvrp{\widehat{\vrp}}
\def\htpsi{\widehat{\psi}}

\def\vrp{\varphi}
\def\m{m}

\def\h{h}
\def\brv{{\bar v}}
\def\brS{{\overline S}}

\def\iv{v^{-1}}
\def\itau{\tau^{-1}}

\def\tD{\mathcal D}
\def\alnu{\al^{\nu}}
\def\Sig{\Sigma}
\def\alhUgm{\al_{U,\gm}^h}
\def\SigzUgm{\Sig_0(U,\gm)}
\def\SigUgm{\Sig(U,\gm)}
\def\alhUgm{\al^h_{U,\gm}}
\def\Uhgm{U^h_{\gm}}

\def\htU{{\widehat U}}
\def\htV{{\widehat V}}

\def\htal{{\widehat \al}}

\def\brBt{{\overline \bt}}

\def\brU{{\overline U}}

\def\brW{{\overline W}}

\def\bral{{\overline \al}}
\def\brbt{{\overline \bt}}
\def\bmfa{\bar \mfa}

\def\EUS{E(U,S)}
\def\EUSU{E(U,S(U))}

\def\dcup{ \; \dot \cup \;}

\def\breta{\bar \eta}
\def\ibt{\bt^{-1}}
\def\igm{\gm^{-1}}
\def\inu{\nu^{-1}}
\def\ilm{\lm^{-1}}
\def\idl{\dl^{-1}}

\def\ipi{\pi^{-1}}
\def\ieta{\eta^{-1}}
\def\irho{\rho^{-1}}
\def\ial{\al^{-1}}
\def\iff{\Leftrightarrow}
\def\sig{\sigma}
\def\simEa{\sim_{E(\mfa)}}
\def\sima{\sim_{\mfa}}
\def\piEa{\pi_{E(\mfa)}}
\def\pia{\pi_{\mfa}}
\def\sm{\setminus}
\def\EUaPh{E(U,\mfa,\Phi)}

\def\EUAgm{E(U,\mfA,\gm)}

\def\onto{\twoheadrightarrow}

\newcommand{\SV}[1]{{\operatorname{SV#1}}}
\newcommand{\D}[1]{{\operatorname{D#1}}}

\newcommand{\etype}[1]{\renewcommand{\labelenumi}{(#1{enumi})}}

\def\EUgm{{E(U, \gm)}}
\def\FUgm{{F(U, \gm)}}
\def\FVgm{{F(V, \gm)}}

\def\FbUgm{{F(\brU, \gm)}}
\def\EbUgm{{E(\brU, \gm)}}

\def\tTU{\tT(U)}
\def\tTV{\tT(V)}

\def\Ea{E(\mfa)}
\def\Eal{E(\al)}
\def\EUa{E_U(\mfa)}

\def\EUb{E_U(\mfb)}
\def\EUmfA{{E(U,\mfA)}}

\def\EUmfB{{E(U,\mfB)}}
\def\EVmfB{{E(V,\mfB)}}
\def\bEUmfB{E(\brU,\bmfB)}

\def\Dis{\operatorname{(Dis)}}

\def\HT{(\operatorname{HT})}
\def\HTp{(\operatorname{HT}')}

\def\STR{\operatorname{STR}}
\def\STROP{\operatorname{STROP}}
\def\strop{\operatorname{STROPH}}
\def\STROPH{\strop}

\def\sring{\operatorname{Sring}}

\def\OST{\operatorname{OST}}

\def\ga{\mathfrak a}

\def\tT{\mathcal T}

\def\tG{\mathcal G}
\def\tH{\mathcal H}

\def\MFC{\operatorname{MFC}}

\def\MFCE{{MFCE}}

\def\endbox{ \hfill\quad\qed}

\def\eroman{\etype{\roman}}
\def\ealph{\etype{\alph}}

\def\pSkip{\vskip 1.5mm \noindent}

\def\sat{\operatorname{sat}}
\def\satUa{\sat_U(\mfa)}
\def\satUb{\sat_U(\mfb)}

\def\mfB{\mathfrak B}

\def\bmfB{\overline{\mfB}}

\def\mfA{\mathfrak A}
\def\mfz{\mathfrak z}
\def\mfq{\mathfrak q}

\def\mfa{\mathfrak a}
\def\mfb{\mathfrak b}
\def\mfc{\mathfrak c}

\def\al{\alpha}
\def\bt{\beta}
\def\gm{\gamma}

\def\tlv{\tilde{v}}
\def\tlx{\tilde{x}}
\def\tlT{\widetilde{T}}
\def\tlL{\widetilde{L}}
\def\tlN{\widetilde{N}}
\def\tlM{\widetilde{M}}
\def\tlU{\widetilde{U}}

\def\tltau{\widetilde{\tau}}
\def\tlal{\widetilde{\al}}
\def\tlvrp{\widetilde{\vrp}}
\def\tlpsi{\widetilde{\psi}}
\def\tlrho{\widetilde{\rho}}
\def\tl0{\widetilde{0}}

\def\ga{\frak a}

\newtheorem{thm}{Theorem} [section]
\newtheorem*{thm*}{Theorem}
\newtheorem{cor}[thm]{Corollary}
\newtheorem{lem}[thm]{Lemma}

\newtheorem{prop}[thm]{Proposition}

\newtheorem*{defn*} {Definition}

\newtheorem*{claim*} {Claim}
\newtheorem*{theorem6.6'} {Theorem 6.6$'$}
\newtheorem{acknowledgment*}[thm] {Acknowledgment}
\newtheorem{example}[thm]{Example}

\newtheorem{addendum}[thm]{Addendum}
\newtheorem{examp}[thm]{Example}

 \newtheorem{lemdef}[thm]{Lemma/Definition}
 \newtheorem{rem}[thm]{Remark}
 \newtheorem{rems}[thm]{Remarks}

 \newtheorem*{remark*}{Remark}
 \newtheorem{defn}[thm]{Definition}
\newtheorem{construction}[thm]{Construction}

\newtheorem{schol}[thm]{Scholium}
\newtheorem{notation}[thm]{Notation}
\newtheorem{notations}[thm]{Notations}
\newtheorem{problem}[thm]{Problem}
\newtheorem{convs}[thm]{Conventions}
\newtheorem{summr}[thm]{Summary}
\newtheorem*{notation*} {Notation}
\newtheorem*{notations*} {Notations}
\newtheorem*{comment*} {Comment}

 \renewcommand{\sectionmark}[1]{}

\newcommand{\Cov}{\operatorname{Cov}}

 \newcommand{\dl}{\delta}

\newcommand{\lm}{\lambda}

 \newcommand{\id}{\operatorname{id}}

 \newcommand{\supp} {\operatorname{supp}}

\newcommand{\z}{\zeta}

\begin{document}

\title[Supertropical Monoids] {Supertropical Monoids:
\\[1mm] Basics, Canonical Factorization, \\[1mm] and  Lifting Ghosts  to Tangibles}
\author[Z. Izhakian]{Zur Izhakian}
\address{Department of Mathematics, Bar-Ilan University, 52900 Ramat-Gan,
Israel}
\email{zzur@math.biu.ac.il}
\author[M. Knebusch]{Manfred Knebusch}
\address{Department of Mathematics,
NWF-I Mathematik, Universit\"at Regensburg 93040 Regensburg,
Germany} \email{manfred.knebusch@mathematik.uni-regensburg.de}
\author[L. Rowen]{Louis Rowen}
 \address{Department of Mathematics,
 Bar-Ilan University, 52900 Ramat-Gan, Israel}
 \email{rowen@macs.biu.ac.il}

\thanks{The research of the first author was supported  by the
Oberwolfach Leibniz Fellows Programme (OWLF), Mathematisches
Forschungsinstitut Oberwolfach, Germany.}

\thanks{The research of the first and third authors have  been  supported  by the
Israel Science Foundation (grant No.  448/09).}

\thanks{The research of the second author was supported in part by
 the Gelbart Institute at
Bar-Ilan University, the Minerva Foundation at Tel-Aviv
University, the Department of Mathematics   of Bar-Ilan
University, and the Emmy Noether Institute at Bar-Ilan
University.}

\thanks{\noindent \underline{\hskip 3cm } \\ File name: \jobname}


\date{\today}




\begin{abstract}
\vskip 0.5cm \emph{Supertropical monoids} are a structure slightly
more general than the supertropical semirings, which have been
introduced and used by the first and the third authors for
refinements of tropical geometry and matrix theory in
\cite{IR1}--\cite{IzhakianRowen2009Resulatants}, and then studied
by us in a systematic way in \cite{IKR1}--\cite{IKR3} in
connection with ``\emph{supervaluations}".

In the present paper we establish a category $\STROP_m$ of
supertropical monoids  by choosing as morphisms the
``\emph{transmissions}", defined in the same way as done in
\cite{IKR1} for supertropical semirings. The previously
investigated category $\STROP$ of supertropical semirings is a
full subcategory of $\STROP_m.$ Moreover, there is associated to
every supertropical monoid $V$ a supertropical semiring $\htV$ in
a canonical way.

A central problem in \cite{IKR1}--\cite{IKR3} has been to find for
a supertropical semiring $U$ the quotient $U / E $ by a
``\emph{TE-relation}", which is  a certain kind of equivalence
relation on the set $U$ compatible with multiplication (cf.
\cite[Definition 4.5]{IKR2}). It turns out that this quotient
always exists in $\STROP_m$. In the good case, that $U / E$ is a
supertropical semiring, this is also the right quotient in
$\STROP.$ Otherwise, analyzing   $(U / E)^\wedge,$ we obtain a
mild modification of $E$ to a TE-relation $E'$ such that $U/ E' =
(U/ E)^ \wedge$ in $\STROP.$

In this way we now can solve various problems left open in
\cite{IKR1}, \cite{IKR2} and gain further insight into the
structure of transmissions and supervaluations. Via supertropical
monoids we also obtain new results on totally ordered
supervaluations and monotone transmissions studied in \cite{IKR3}.
\end{abstract}

\maketitle

\tableofcontents

\baselineskip 14pt

\numberwithin{equation}{section}

\section*{Introduction}
To a large extent the algebra underpinning present day tropical
geometry is based on the notion of a (commutative)
\textbf{bipotent semiring}. Such a semiring $M$ is a totally
ordered monoid under multiplication with smallest element $0$, the
addition being given by $x+y = \max(x,y)$, cf. \cite[\S1]{IKR1}
for details. In logarithmic notation, which most often is used in
tropical geometry, bipotent semirings appear as totally ordered
additive monoids with absorbing element $-\infty$. The primordial
object here is the bipotent semifield $T(\Real) = \Real \cup
\{-\infty \},  $ cf. e.g.  \cite[\S1.5]{IMS}.

In \cite{I} the first author introduced a cover of $T(\Real)$,
graded by the multiplicative monoid $(\Int_2, \cdot \; )$, which
was dubbed the \emph{extended tropical arithmetic}. Then, in
\cite{IR1} and \cite{IR2}, this structure has been amplified to
the notion of a \textbf{supertropical semiring}. A supertropical
semiring $U$ is equipped with a ``ghost map" $\nu := \nu_U: U \to
U$, which respects addition and multiplication and is idempotent,
i.e., $\nu \circ \nu = \nu$. Moreover, in this semiring $a +a =
\nu(a)$  for every $a \in U$ (cf. \cite[\S3]{IKR3}). This rule
replaces the rule $a+a = a$ taking place in the usual max-plus (or
min-plus) arithmetic. We call $\nu(a)$ the ``\textbf{ghost}" of
$a$ and we term  the elements of $U$, which are not ghosts,
``\textbf{tangible}". (The element $0$ is regarded both as
tangible and ghost.) $U$ then carries a multiplicative idempotent
$e = e^2$  such that $\nu(a) = ea$ for every $a \in U$. The
image~$eU$ of the ghost map, called the \textbf{ghost ideal} of
$U$, is itself a bipotent semiring.

Supertropical semirings allow a refinement of valuation theory to
a theory of ``supervaluations", the basics of which can be found
in \cite{IKR1}-\cite{IKR3}. Supervaluations seem to be able to
provide an enriched version of tropical geometry, cf. \cite[\S9.
\S11]{IKR1} and \cite{IR1}. We recall the initial definitions.

An \textbf{\m-valuation} (= monoid valuation) on a semiring $R$ is
a multiplicative map $v: R \to M $ to a bipotent semiring $M$ with
$v(0) =0 $, $v(1) = 1$,
 and
 $$ v(x+y) \ds \leq v(x) + v(y) \quad [ = \max(v(x), v(y))],$$
 cf. \cite[\S2]{IKR1}. We call $v$ a \textbf{valuation} if in
 addition the semiring $M$ is \textbf{cancellative}, by which we
 mean that $M \sm \00 $ is closed under multiplication and is a
 cancellative monoid in the usual sense. If $R$ happens to be a
 (commutative) ring, these valuations coincide with the valuations
 of rings defined by Bourbaki \cite{B} (except that we switched from
  additive notation there to multiplicative notation here).

 Given an \m-valuation $v: R \to M $ there exist multiplicative
 mappings $\vrp: R \to U$ into various supertropical semirings $U$
 with $\vrp(0) = 0$, $\vrp(1) = 1$, such that $M$ is the ghost
 ideal of $U$ and $\nu_U \circ \vrp = v.$ These  are the
 \textbf{supervaluations}  covering $v$, cf. \cite[\S4]{IKR1}.

 The supervaluations lead us to the ``right" class of maps between
 supertropical semirings $U$, $V$ which we have to admit as morphisms to obtain
 the category $\STROP$
 of supertropical semirings (formally introduced in \cite[\S1]{IKR2}).
 These are the ``transmissions". A
 \textbf{transmission} $\al: U \to V$ is a multiplicative map with
 $\al(0) = 0$,  $\al(1) = 1$,  $\al(e_U) = e_V$, whose restriction
 $\gm: eU \to eV$ to the ghost ideals is a semiring homomorphism.
 It turned out in \cite[\S5]{IKR1} that the transmissions from $U$
 to $V$ are those maps $\al: U \to V$ such that for every
 supervaluation $\vrp: R \to U$ the map $\al \circ \vrp : R \to V$
 is again a supervaluation.

Every semiring homomorphism $\al: U \to V$  is a transmission, but
there also exist transmissions which do not respect addition. This
causes a major difficulty for working in $\STROP.$ A large part of
the papers \cite{IKR1}, \cite{IKR2} has been devoted to
constructing various equivalence relations $E$ on a supertropical
semiring $U$ such that the map $\pi_E: U \to U/E,$ which sends
each $x\in U$ to its $E$-equivalence class $[x]_E$, induces the
structure of a supertropical semiring on the set $U/E$  which
makes $\pi_E$ a transmission. We call such an equivalence relation
$E$ \textbf{transmissive}.

It seems difficult to characterize transmissive equivalence
relations in an axiomatic way.  In \cite[Proposition 4.4]{IKR2}
three axioms TE1--TE3 have been provided, which obviously have to
hold, but then, adding a fourth axiom, we only characterized those
transmissive equivalence relations on $U$, where the ghost ideal
of $U/E$ is cancellative \cite[Theorem~4.5]{IKR2}. Dubbing the
equivalence relations obeying axioms TE1--TE3 ``TE-relations", we
had to leave the following problem open in general:
\begin{itemize}
    \item[$(*)$] When is a TE-relation $E$ on a
    supertropical semiring $U$ transmissive?
\end{itemize}

The problem seems to be relevant since there exist natural classes
of \m-valuations \\ $v: R \to M,$  where the bipotent semiring $M$
has no reason to be cancellative, cf. \cite[\S1]{IKR3}. For~$R$ a
ring such \m-valuations already appeared in the work of
Harrison--Vitulli \cite{HV} and D. Zhang~\cite{Z}.

The present paper  gives a solution to the problem $(*)$ just
described. We  introduce a new category $\STROP_m$ containing the
category $\STROP$ of supertropical semirings as a full
subcategory. The objects of $\STROP_m$, called
``\textbf{supertropical monoids}", are multiplicative monoids $U$
with an absorbing element $0$, an idempotent $e\in U,$ and a total
ordering on the monoid $eU,$ which makes  $eU$ a bipotent
semiring.

 In a supertropical monoid it is natural to speak about
tangibles and ghosts in the same way as in supertropical
semirings. Every supertropical semiring can be regarded as a
supertropical monoid, of course. Conversely,  since addition in a
supertropical semiring $U$ is determined by multiplication and the
idempotent $e$ (cf. \cite[Theorem 3.11]{IKR3}), we can turn a
supertropical monoid $U$ into a supertropical semiring in at most
one way, and then say that~$U$ ``is" a supertropical semiring. If
$U$ and $V$ are supertropical monoids, the definition of a
transmission $\al:U \to V$, as given above for $U$, $V$
supertropical semirings, still makes sense, and these
``transmissions" between supertropical monoids (cf. Definition
\ref{defn1.4}) are taken   as the morphisms in the category
$\STROP_m$.

The axioms TE1--TE3 mentioned above also  make perfect sense for
an equivalence relation~$E$ on a supertropical monoid $U$ (cf.
Definition \ref{defn1.6} below). Thus, such a relation~$E$ will
again be called a \textbf{TE-relation}. But we have the important
new fact that for a TE-relation~$E$ on a supertropical monoid $U$
the quotient $U/E$ \textbf{always} exists in the category
$\STROP_m$. More precisely, the set $U/E$ can be equipped  with
the structure of a supertropical monoid in a unique way, such that
the map $\pi_E$ is a transmission (cf. Theorem \ref{thm1.7}
below).

The solution of the problem $(*)$ from above now reads as follows
(Scholium \ref{schol1.12} below): If $U$ is a supertropical
semiring and $E$ is a TE-relation on $U$, then $E$ is transmissive
iff the quotient  $U/E$ in $\STROP_m$ is a supertropical semiring.

We will provide a necessary and sufficient condition that a given
supertropical monoid $U$ is a supertropical semiring (Theorem
\ref{thm1.2} below). From this criterion it is immediate that $U$
is such a semiring  if the bipotent semiring $U$ is cancellative,
but there also are  other cases,  where this holds. \pSkip

A bipotent semiring $M$ may be viewed as a supertropical semiring,
all of whose elements are ghosts. (This is the case $1 = e.$) Thus
the category $\STROP_m /M$ of supertropical monoids over $M$ may
be viewed as the category of supertropical monoids $U$ with a
fixed ghost ideal $eU = M.$ Then the morphisms of $\STROP_m /M$
are the transmissions $\al: U \to V$ with $\al(x) = x$ for all $x
\in M.$ We call the surjective transmissions over $M$
\textbf{fiber contractions} (over $M$), as we did for
supertropical semirings \cite[\S6]{IKR1}. We note in passing that
if $\al:U \onto V$ is a fiber contraction and $U$ is a
supertropical semiring, then $V$ is again a supertropical semiring
(cf. Theorem \ref{thm1.5} below), and $\al$ is a semiring
homomorphism \cite[Propoistion 5.10.iii]{IKR1}.

It turns out that for every supertropical monoid $U$ there exists
a fiber contraction $\sig_U: U \to \htU$  with $\htU$ a
supertropical semiring, called the \textbf{supertropical semiring
associated} to $U$, such that every fiber contraction $\al:U \onto
V$ factors through $\sig_U$ (in a unique way), $\al = \bt \circ
\sig_U$ with $\bt:\htU \to V$ again a fiber contraction. In more
elaborate terms, $\STROP /M$ is a full reflective subcategory of
$\STROP_m /M$, cf. \cite[p. 79]{F}, \cite[1.813]{FS}.

The reflections $\sig_U: U \to \htU$ turn out to be useful for
solving problems of universal nature for supertropical semirings
and supervaluations. The strategy is, first to solve such a
problem in $\STROP_m$, which often is easy, and then to employ
reflections to obtain a solution in $\STROP.$ Major instances for
this are provided by Theorems \ref{thm4.6} and \ref{thm6.9} below.
\pSkip

A large part of  the paper is devoted to the factorization of
transmissions into appropriately defined  ``basic" transmissions.
Let $\al : U \onto V$ be a surjective transmission with $U,V$
supertropical monoids, and let $\al^\nu: = \gm : M \onto N$ denote
the ``ghost part" of $\al$, i.e., the semiring homomorphism
obtained from $\al$  by restriction to the ghost ideals $M:= e_U
U$, $N : =e_V V$. Then there exists an essentially unique
factorization
\begin{equation}\label{eq:0.1}
    \xymatrix{
    \al:U    \ar[r]_{\lm} & U_1  \ar[r]_{\bt} & V_1 \ar[r]_{\mu} & V,  \\
}
\end{equation} for some supertropical monoids $U_1$ and $V_1$,
with $\lm $ and $\mu$ fiber contractions of certain types over $M$
and $N$ respectively and $\bt$ a so called ``strict
ghost-contraction", which means that $\bt$ restricts to a
bijection from the set $\tT(U_1)$ of non-zero tangible elements of
$U$ to $\tT(U_1)$, while $\bt^\nu = \gm,$ cf. Theorem
\ref{thm2.9}. (Notice that $\gm$ has convex fibers in $M$, since
$\gm$ respects the orderings of $M$ and $N$. These convex sets are
contracted by $\gm$ to one-point sets, hence the name
``ghost-contraction". ``Strict" alludes to the property that no
element of $\tT(U)$ is sent to a ghost in $V$.)

From \eqref{eq:0.1} we then obtain a factorization
\begin{equation}\label{eq:0.2}
\xymatrix{
    \al:U    \ar[r]_{\lm} & \brU  \ar[r]_{\bt} & W \ar[r]_{\mu} & \brW \ar[r]_{\rho} & V  \\
}
\end{equation}
which is really unique. Here $\lm, \bt, \mu$ are transmissions of
the same types as before but are normalized to maps $\pi_E$ given
by certain TE-relations $E$ on $U$, $\brU$, $W$ respectively,
which are uniquely determined by $\al$, and $\rho$ is an
isomorphism over $N = eV.$ This is the ``\textbf{canonical
factorization}" of $\al$, appearing in the title of the paper
(Definition 2.12). The transmissions $\lm, \bt, \mu, \rho$ are
called the \textbf{canonical factors} of $\al.$

In \S3 we make explicit how the canonical factors of a composition
$\al_2 \circ \al_2$ of two transmissions $\al_1: U_1 \to U_2$,
$\al_2: U_2 \to U_3$ can be obtained from the canonical factors of
$\al_1$ and~$\al_2.$

Our primary interest is not in supertropical monoids but in
supertropical semirings. In this respect the following result is
useful: If $U$ and $V$ are supertropical semirings, then in the
canonical factorization \eqref{eq:0.2} all three supertropical
monoids $\brU, W, \brW $ are again supertropical semirings, and
thus all canonical factors are morphisms in $\STROP$ (Theorem
\ref{thm2.9}).

Besides  $\STROP$, the category $\STROPH$ deserves interest, whose
objects are  again the supertropical semirings but whose morphisms
are only the semiring homomorphisms. (Thus $\STROPH$ is a
subcategory of $\STROP$ and a full  subcategory of the category of
semirings.) In \S5 we introduce a subcategory $\STROPH_m$ of
$\STROP_m$ which turns out to be equally useful for working in
$\STROPH$, as $\STROP_m$ has proved to be useful for $\STROP.$

The objects of $\STROPH_m$  are again the supertropical monoids,
but the morphisms are suitably defined
``\textbf{\h-transmissions}", which are designed in such a way
that an \h-transmission $\al:U \to V$ between supertropical
semirings is a semiring homomorphism, cf. Definition
\ref{defn5.1}. Thus $\STROPH$ is a full subcategory of
$\STROPH_m.$ Again it turns out that for a given bipotent semiring
$M$ the category $\STROPH / M$ is reflective in $\STROPH_m / M$
(Corollary~\ref{cor5.10}).

If $\al: U \to V$ is a surjective \h-transmission, then the
canonical factors of $\al$ are  again \h-transmissions (Theorem
\ref{thm5.11}). It follows that, if $\al: U \to V$ is a surjective
semiring homomorphism, the whole canonical factorization runs in
$\STROPH$.

In \S6 we study supertropical monoids which have a total ordering
defined to be compatible with the supertropical monoid structure
in a rather obvious way (Definition \ref{defn61.1}). We call them
\textbf{ordered supertropical monoids} (= \textbf{OST-monoids},
for short). It turns out that the underlying supertropical monoids
of  an $\OST$-monoid is a supertropical semiring (Theorem~
\ref{thm61.4}).

The ``right" morphisms between $\OST$-monoids are the
transmissions compatible with the given orderings, called
\textbf{monotone transmissions}. It turns out that every monotone
transmission is a semiring homomorphism (Theorem \ref{thm61.7}). A
major result now is that, given a monotone transmission $\al: U
\to V$, all factors of the canonical factorization \eqref{eq:0.2}
of~$\al$ can be interpreted as monotone transmissions. More
precisely, there exist unique total orderings on $\brU, W, \brW$,
which make $\brU, W, \brW$ $\OST$-monoids and all factors $\lm,
\bt, \mu, \rho$ monotone transmissions (Theorem \ref{thm61.14}).

In the last sections \S\ref{sec:6}--\S\ref{sec:8} of the paper
``\emph{\m-supervaluations}" play the leading role. Given an
\m-valuation $v: R \to M$ on a semiring $R$, an
\textbf{\m-supervaluation $\vrp: R \to U$ covering} $v$ is defined
in a completely analogous way as has been indicated above for a
supervaluation. The only difference is that now $U$ is a
supertropical monoid instead of a supertropical semiring
(Definition \ref{defn6.1}).

The morphisms in $\STROP_m$ are adapted to the \m-supervaluations,
as the morphisms in $\STROP$ were adapted to the supervaluations,
to wit, a map $\al: U \to V$ between supertropical monoids is a
transmission iff for every \m-supervaluation $\vrp:R \to U$ the
map $\al \circ \vrp$ is again an \m-supervaluation. (This has not
been detailed in the paper.)

In order to avoid discussions about ``equivalence" of
\m-valuations we now tacitly assume, without essential loss of
generality, that all occurring \m-supervaluations  $\vrp:R \to U$
are ``surjective", i.e., $U = \vrp (R) \cup e
\vrp(R)$.\footnote{Although this does not mean surjectivity in the
usual sense, there is no danger of confusion since a
supervaluation $\vrp:R \to V$ can hardly ever be surjective as a
map except in the degenerate case $V=M$.} Given such
\m-supervaluations $\vrp : R \to U$, $\psi: R \to V$, w say that
$\vrp$ \textbf{dominates} $\psi$, and write $\vrp \geq \psi$, of
there exists a (unique) transmission $\al: U \to V$ with $\psi =
\al \circ \vrp.$

In \S\ref{sec:6} we construct the \textbf{initial
\m-supervaluation} $\vrp_v^0 : R \to U(v)^0$ covering  a given
\m-supervaluation $v: R \to M$, which means that $\vrp^0_v$
dominates all other \m-supervaluations covering $v$. It then is
immediate that
$$ \vrp_v := \sig _{U(v)^0} \circ \vrp^0_v: R \ds \to (U(v)^0)^\wedge$$
is the  initial supervaluation covering $v$, i.e., dominates all
supervaluations covering $v$ (Theorem \ref{thm6.9}).

Already in \cite[\S7]{IKR1} we could prove that an initial
supervaluation $\vrp_v$ covering $v$ exists, but obtained  an
explicit description of $\vrp_v$ only in the case that $M$ is
cancellative, while now we obtain an explicit description of
$\vrp_v$ in general (Scholium \ref{schl6.11}). (N.B. If $M$ is
cancellative, then $\vrp_v = \vrp_v^0$.)

More generally, if $\vrp: R \to U$ is an \m-supervaluation
covering $v$, then $$\htvrp:= \sig_U \circ \vrp : R \to \htU$$ is
a supervaluation covering $v$, and $\htvrp \geq \psi$ for any
other supervaluation $\psi$  covering $v$ with $\vrp \geq \psi$.
We call $\vrp$ \textbf{tangible}, if $\vrp(R) \subset \tT(U) \cup
\00$. It turns out that $\vrp^0_v$ is always tangible, but if
$U(v)^0$ is not a supertropical semiring, i.e., $\vrp^0_v \neq
\vrp_v,$ then $\vrp_v$ is not tangible, and this implies that no
supervaluation covering $v$ is tangible.

The last two sections \S\ref{sec:7}, \S\ref{sec:8} are motivated
by our interest to put a supervaluation $\vrp: R \to U$ covering
$v$ to use in tropical geometry. It will be relevant to apply
$\vrp$ to the coefficients of a given polynomial $f(\lm) \in
R[\lm]$ in a set $\lm$ of $n$ variables (or a Laurent polynomial),
and to study the supertropical root sets and tangible components
of polynomials $g(\lm) \in F[\lm]$ in $F^n$ (cf. \cite[\S5,
\S7]{IR1}) obtained from $f^\vrp(\lm) \in U[\lm]$ by passing from
$U$ to various supertropical semifields $F$. For this purpose it
will be important to have some control on the set
$$ \{ a \in R \ds |  \vrp(a) \in M \} \ds  =  \{ a \in R \ds |  \vrp(a) = v(a) \}.$$
Given an \m-supervaluation $\vrp: R \to Y$ covering $v: R \to M$,
we construct a tangible \m-supervaluation $\tlvrp: R \to \tlU,$
which is minimal with $\tlvrp \geq \vrp$ (Theorem \ref{thm7.10}).
 In \S\ref{sec:8} we then classify
the \m-supervaluations $\psi$ with $\vrp \leq \psi \leq \tlvrp,$
called the \textbf{partial tangible lifts} of $\vrp.$ They are
uniquely determined by their \textbf{ghost value sets}
$$ G(\psi) := \psi (R) \cap M,$$
cf. Theorem \ref{thm8.4}. These are ideals of the semiring $M,$
and all ideals $\mfa \subset G(\vrp)$ occur in this way (Theorem
\ref{thm8.7}). Unfortunately the ghost value set $G(\psi)$  does
not control the set $\{ a \in R \ds | \psi(a) \in M\} $
completely.  We can only state that this set  is contained in
$\iv(G(\psi))$.

If $\vrp$ is a supervaluation,  then $\chvrp := (\tlvrp)^\wedge$
is the supervaluation, which is a partial tangible lift of $\vrp$
having smallest ghost value set.
\begin{notations*}
Given sets $X,Y$ we mean by $Y \subset X$ that $Y$ is a subset of
$X$, with $Y  = X$ allowed. If $E$ is an equivalence relation on
$X$ then $X/E$ denotes the set of $E$-equivalence classes in $X$,
and $\pi_E: X \to X/E$ is the map which sends an element $x$ of
$X$ to its $E$-equivalence class, which we denote by $[x]_E$. If
$Y \subset X$, we put $Y/E := \{[x]_E  \ds | x \in Y\}.$

If $U$ is a supertropical semiring, we denote the sum $1+1$ in $U$
by $e$, more precisely by $e_U$ if necessary. If $x \in U$ the
\textbf{ghost companion} $ex$ is also denoted by $\nu(x)$, and the
\textbf{ghost map} $U \to eU$, $x \mapsto \nu(x)$, is denoted by
$\nu_U$. If $\al: U \to V$ is a transmission, then the semiring
homomorphism $eU \to eV$ obtained from $\al$ by restriction is
denoted by  $\al^\nu$ and is called the \textbf{ghost part} of
$\al$. Thus $\al^\nu \circ \nu_U = \nu_V \circ \al$.

$\tT(U)$ and $\tG(U)$ denote the sets of tangible and ghost
elements of $U$, respectively, cf. \cite[Terminology 3.7]{IKR1}.
We put $\tT(U)_0 := \tT(U) \cup \00.$

If $v : R \to M $ is an \m-valuation we call the ideal $v^{-1}(0)$
of $R$ the \textbf{support} of $v$, and denote it by $\supp(v)$.
\end{notations*}

\section{Supertropical monoids}\label{sec:1}

\begin{defn}\label{defn1.1} A \textbf{supertropical monoid} $U$ is
a  monoid $(U, \cdot \, )$ (multiplicative notation) which has an
absorbing element $0 := 0_U$, i.e., $0 \cdot x = 0$ for every $x
\in U$, and a distinguished central idempotent $e := e_U$ such
that the following holds:
$$ \forall x \in U: \quad ex =0 \ \Rightarrow \ x = 0.$$
Further a total ordering, compatible with multiplication, is given
on the submonoid $M := eU$ of $U$.

We then regard $M$ as a bipotent semiring in the usual way
\cite[\S1]{IKR1}.
\end{defn}

If $U$ is a supertropical monoid, we would like to enrich $U$ by a
composition $ U \times U \overset{+}{\longrightarrow} U$ extending
the addition on $M$, such that $U$ becomes a supertropical
semiring with $$1_U + 1_U = e_U.$$ We are forced to define the
addition on $U$ as follows ($x,y \in U$), cf. \cite[Theorem
3.11]{IKR1}:
$$ x+y = \left\{
\begin{array}{cccc}
  y &  & \text{if} & ex < ey, \\
  x &  & \text{if} & ex > ey , \\
  ex &  & \text{if} & ex= ey. \\
\end{array} \right.
$$
If this addition obeys the associativity and distributivity laws,
we say that the supertropical monoid $U$ ``is'' a semiring. In the
commutative case we have the following criterion.
\begin{thm}\label{thm1.2}
A supertropical commutative  monoid is a semiring iff the
following holds:
\begin{align*}
\Dis:  \quad &  \forall x,y,z \in U: \text{ If } 0 < ex < ey, \text{ but}\\
 & exz = eyz,  \text{ then } yz = eyz
 \text{ (i.e., } yz \in eU).
\end{align*}
In this case the semiring $U$ is supertropical.
\end{thm}

\begin{proof}

Let $x,y,z \in U$ be given, Obviously, $x +y = y +x $ and $x +0 =
x$, and it is easily checked that $(x+y) + z = x + (y+z)$. It
remains to investigate when we have
 \begin{equation}\renewcommand{\theequation}{$*$}\addtocounter{equation}{-1}\label{eq:str.1}
  (x+y)z = xz + yz.
\end{equation}

We assume without loss of generality that $ex \leq ey$. If $ex=0$,
then $x=0$ and $(*)$ is true. If $ex =e y$, then $exz = eyz$,
hence $x +y = ey$ and $xz+ yz = eyz$. Thus $(*)$ is true again.

We are left with the case that $0 < ex< ey$. Then $x+y =y$ and
$exz \leq eyz$. If  $exz < eyz$, then $xz + yz = yz$, and $(*)$ is
true. But if $exz = eyz$, then $xz+yz = eyz$, while $(x+y)z =yz$.
Thus $(*)$ holds iff $yz = eyz$.

We conclude that $(*)$ holds for all triples  $x,y,z$ iff
condition $\Dis$ is fulfilled.
\end{proof}

\begin{rem}\label{rmk:1.13} When $U$ is not commutative, we have
an analogous  result.  We just need to add the same condition
$\Dis$ for the monoid $U^{\opp}$ obtained from $U$ by changing the
multiplication $(x,y) \mapsto xy$ to $(x,y) \mapsto yx$; i.e.,
\begin{align*}
\Dis':  \quad &  \forall x,y,z \in U: \text{ If } 0 < ex < ey, \text{ but}\\
 & ezx = ezy,  \text{ then } zy = ezy.
\end{align*}
\end{rem}

Given an element $x$ of a supertropical monoid $U$, we call $ex$
the \textbf{ghost} of $x$, and we denote the \textbf{ghost map} $U
\to eU$,  $x \mapsto ex$,  by $\nu_U$, as we did before for $U$ a
supertropical semiring.

By a (two sided) \textbf{ideal} $\ga$ of a supertropical monoid
$U$ we mean a monoid ideal of $U$, i.e., a nonempty subset $\mfa$
of $U$ with $U \cdot \mfa \subset \mfa$ and $ \mfa \cdot U \subset
\mfa$. Notice that in the case that $U$ is a supertropical
semiring, such a set $\mfa$ is indeed an ideal of the semiring $U$
in the usual sense, cf. \cite[Remark 6.21]{IKR2}. We call $eU$ the
\textbf{ghost ideal} of $U$.

Many more definitions in \cite{IKR1} and \cite{IKR2} retain their
sense  if we replace the supertropical semirings by supertropical
monoids, in particular the following one.
\begin{defn}\label{defn1.3}
Let $U$ and $V$ be supertropical monoids. We call a map $\al: U
\to V$ a \textbf{transmission}, if the following holds (cf.
\cite[\S5]{IKR1}): \begin{alignat*}{2}
&TM1: \quad&& \alpha(0)=0,\\
&TM2: \quad &&\alpha(1)=1,\\
 &TM3: \quad &&\forall x,y\in U:\quad
\alpha(xy)=\alpha(x)\alpha(y),\\
&TM4:\quad&&\alpha(e_U)=e_V,\\
&TM5:\quad &&\forall x,y\in eU: \quad
 x \leq y  \Rightarrow \alpha(x)  \leq \alpha(y).
 \end{alignat*}
(N.B. $\al$ maps $eU$ to $eV$ due to TM3 and TM4.) Notice that
this means that $\al$ is a monoid homomorphism sending $0$ to $0$
and $e$ to $e$, which restricts to a  homomorphism $\gm: eU \to
eV$ of bipotent semirings.
 We then call $\gm$ the \textbf{ghost part} of~$\al$,
and write $\gm = \al^\nu$. We also say that~$\al$ \textbf{covers}
$\gm$ (as we did in \cite{IKR1} for $U$, $V$ commutative
supertropical semirings). Notice that $\al(U)$ is a supertropical
submonoid of $V$ in the obvious sense.
\end{defn}

We introduce two sorts of ``kernels'' of transmissions.

\begin{defn}\label{defn1.4}
Let $\al: U \to V$ be a transmission between supertropical
monoids.
\begin{enumerate} \ealph
    \item The \textbf{zero kernel} of $\al$ is the set
    $$ \mfz_\al := \{ x \in U \ | \ \al(x) =0\}.$$
    \item The \textbf{ghost kernel} of $\al$ is the set
    $$ \mfA_\al := \{ x \in U \ | \ \al(x) \in eV \}.$$
\end{enumerate}
\end{defn}

These sets are ideals of $U$, and $M \cup  \mfz_\al \subset
\mfA_\al$. If $U$ is a  semiring, then $M \cup \mfz_\al =  M +
\mfz_\al$, (cf. \cite[Remark 6.21]{IKR2}). If $ \mfA_\al = M$, we
say  that $\al$ has \textbf{trivial ghost kernel}, and if~$
\mfz_\al = \{ 0 \}$,  we say that $\al$ has \textbf{trivial zero
kernel}.

\begin{thm}\label{thm1.5}
Let $\al: U \to V$ be a transmission between supertropical
monoids, which is injective on the set $(eU) \setminus \{ 0\}$.
\begin{enumerate} \eroman
    \item If $\al$ has a trivial ghost kernel, and if $V$ is a
     semiring, then $U$ is a  semiring.
    \item If $\al$ is surjective, and if $U$  is a
    semiring, then $V$ is a  semiring.
\end{enumerate}
\end{thm}

\begin{proof}
Again we prove this for  commutative monoids, leaving the obvious
modifications in the noncommutative case to the interested reader.

We use the criterion for a supertropical monoid to be a  semiring
given in Theorem~\ref{thm1.2}.

(i): Let $x,y,z \in U$ with $0 < ex < ey$ and $exz = eyz$. Then
$$ e\al(x)= \al (ex) < \al(ey) = e \al(y),$$
since $\al$ is injective on $(eU) \setminus \{ 0 \}$, and
$$ e \al(x) \al(z) =  e \al(y) \al(z).$$
Since $V$ is a semiring, we  deduce that $\al(yz) =  \al(y) \al(z)
\in eV$. Since $\al $ has a trivial ghost kernel, it follows that
$yz \in eU$, as desired. \pSkip

(ii): Let $x,y,z \in U$ with
$$ 0 < e \al(x) < e \al(y) \quad \text{and} \quad e \al(x) \al(z) =  e \al(y) \al(z).$$
We are done if we verify that $ \al(y) \al(z) \in e V$. We have $0
< \al (ex) < \al (ey)$ and $\al(exz) = \al(eyz)$.The inequalities
imply $0 < ex < ey$.
\begin{description}

    \item[{Case I}] $\al(eyz) = 0$.
    Now $e \al(yz) = 0$, hence $\al(yz)= 0$, hence $\al(y) \al(z)
    =0.$ \pSkip
    \item[{Case II}] $\al(eyz) \neq 0$.
    Now  $exz \neq 0$ and $eyz \neq 0 $. Since $\al$ is injective on
    $(eU) \setminus \{ 0 \} $, it follows that $exz = eyz$. Since  $0 < ex <
    ey$ and $U$ is a semiring, we conclude that $yz \in eU$, hence
    $ \al(y)\al(z) = \al (yz) \in eV$.

\end{description}
    Thus $\al(y) \al(z) \in eV$ in both cases.
\end{proof}

\begin{defn}\label{defn1.6} If $U$ is a supertropical monoid, we
call an equivalence relation $E$ on the set~$U$ a
\textbf{TE-relation}, if the following holds (cf.
\cite[\S4]{IKR2}):
\begin{alignat*}{2}
&TE1: \quad&& \text{$E$  is multiplicative,
i.e., } \forall x,y,z \in E:\\
& && x \sim_E y \dss \Rightarrow xz \sim_E yz \ , zx \sim_E zy  \\
 &TE2: \quad && \text{The equivalence relation $E | M $ is order
compatible, i.e.:} \\
& & & \text{If $x_1,x_2,x_3, x_4 \in M$ and $x_1 \leq x_2$, $x_3 \leq x_4$,
$x_1 \sim_E x_4$, $x_2 \sim_E x_3$,} \\
& & & \text{then $x_1 \sim_E x_2$. } \text{(Hence all $x_i$ are $E$-equivalent.)} \\
&TE3: \quad&& \text{If $x \in U$ and $ex \sim_E 0$, then $x \sim_E
0$.}
 \end{alignat*}
\end{defn}
We have the following almost trivial but important fact.

\begin{thm}\label{thm1.7}
Let $U $ be a supertropical monoid and $E$ a TE-relation on $U$.
Then the set $U/E$ of equivalence classes carries a unique
structure of a supertropical monoid such that the map
$$ \pi_E: U \to U/E, \qquad x \mapsto [x]_E,$$
is a transmission.
\end{thm}

\begin{proof} This is just some universal algebra. We are forced to define the multiplication on the set
$\brU := U/E$ by the rule ($x,y \in U$)
$$ [x]_E \cdot [y]_E = [xy]_E.$$
This makes sense since the equivalence relation  $E$ is
multiplicative. Now $\brU$ is a  monoid with absorbing element
$0_\brU := [0_U]_E$. We are further forced to take as
distinguished idempotent on $\brU$ the element $e_\brU:= [e_U]_E$.
Clearly
$$ e_\brU \brU = M/E := \{ [x]_E \ |  \ x \in M\}.$$
Finally, we are forced to choose on the submonoid $M/E$ of $U/E$
the total ordering given by ($x,y \in M$)
$$ [x]_E \leq [y]_E \dss \Leftrightarrow x \leq y.$$
This total ordering is well-defined since the restriction  $E|M$
of $E$ to $M$ is order compatible (cf. \cite[\S2]{IKR2}).

It is now evident that $\brU$ has become a supertropical monoid
and $\pi_E$ a transmission.
\end{proof}

\begin{rem}\label{rem1.8}
Conversely, if  $\al: U \to V$ is  a transmission from $U$ to a
supertropical monoid $V$, then the equivalence relation $E(\al)$
is TE, and the map $[x]_{E(\al)} \mapsto \al(x)$ is an isomorphism
from the supertropical monoid $U/E(\al)$ onto the (supertropical)
submonoid $\al(U)$ of $V$.
\end{rem}

\begin{example} Let $U$ be a supertropical monoid and $M:= eU$. As
in the case of supertropical commutative semirings
\cite[\S6]{IKR1} we define an \textbf{MFCE-relation} on $U$ as an
equivalence relation $E$ on~$U$, which is multiplicative, and is
fiber conserving, i.e., $x \sim_E y \Rightarrow ex = ey$. Then we
have an obvious identification  $M/E = M$, and $E$ is a
TE-relation.
\end{example}

The functorial properties of transmissions between supertropical
semirings stated in \cite[Proposition 6.1]{IKR1} remain true if we
admit instead supertropical monoids, and can be proved in exactly
the same way. Thus we get:
\begin{prop}\label{prop1.9}
Let $\al: U \to V$ and $\bt: V \to W$ be maps between
supertropical monoids.
\begin{enumerate} \eroman
    \item If $\al $ and $\bt $ are transmissions, then $\bt \al$
    is a transmission.

    \item If $\al$ and $\bt\al$ are transmissions and $\al$ is
    surjective, then $\bt$ is a transmission.
\end{enumerate}
\end{prop}

Starting from now we  assume that \textbf{all occurring
supertropical monoids are commutative}. But we mention that all
major results  to follow can be established also for
noncommutative monoids with obvious modifications of the proofs
(in a similar way as indicated in Remark \ref{rmk:1.13}). This
will save space and hopefully help the reader not to get
distracted from the central ideas of the paper.  At the time
being,  the commutative case suffices for the applications we have
in mind (cf. the Introduction).

We define the \textbf{category of supertropical monoids}
$\STROP_m$ as follows: the objects of $\STROP_m$  are the
(commutative) supertropical monoids, and the morphisms are the
transmissions between them. $\STROP_m$ contains the category
$\STROP$ of supertropical semirings as a full subcategory.

\begin{schol}\label{schol1.12}
Let $U$  be a supertropical semiring and $E$ a TE-relation on $U.$
Then the map $\pi_E: U \to U/E$ from $U$ to the supertropical
monoid $U/E$ is a morphism in $\STROP_m.$ Since $\STROP$ is full
in $\STROP_m$, it follows that $\pi_E$ is a morphism in $\STROP_m$
iff the supertropical monoid $U/E$ is a semiring. This means in
terms of \cite[\S4]{IKR2}, that the TE-relation $E$ is
transmissive iff the supertropical monoid $U/E$ is a semiring.
\end{schol}

We  define \textbf{initial transmissions} and \textbf{pushout
transmissions} in $\STROP_m$ as we defined such transmissions in
\cite[\S1]{IKR2} in the category $\STROP$. Just repeat
\cite[Definition 1.2]{IKR2} and \cite[Definition 1.3]{IKR2},
respectively, replacing everywhere the word ``supertropical
semiring'' by ``supertropical monoid''.

The pleasant news now is that in $\STROP_m$ the pushout
transmission  exists for any supertropical monoid $U$ and
surjective homomorphisms $\gm$ from  $M := eU$ to a bipotent
semiring~$N$, and that it has the same explicit description as
given in \cite[Theorem 1.11]{IKR2} and \cite[Example 4.9]{IKR2}
(in the category $\STROP$) if $N$ is cancellative.

More precisely the following holds and can be proved by the same
arguments as used in \cite[Example 4.9]{IKR2} and the proofs of
\cite[Theorems 1.11 and 4.14]{IKR2}.

\begin{thm}\label{thm1.10} Let $U$ be a supertropical monoid and
$\gm$ a homomorphism from $M:=eU$ onto a (bipotent) semiring
\footnote{Notice that a homomorphic image of a bipotent semiring
is again bipotent.} $N$. We obtain a TE-relation $F(U,\gm)$ on $U$
by decreeing for $x,y \in U$:
$$ x \sim_{F(U,\gm)}y  \quad \Leftrightarrow \quad \left\{
\begin{array}{lll}
  \text{either} &  &  x=y, \\
  \text{or} &  &  x=ey,\ y =ey, \ \gm(ex) = \gm(ey), \\
  \text{or} &   & \gm(ex) = \gm(ey) =0. \\
\end{array}%
\right.$$ The map $$\pi_{\FUgm} \twoheadrightarrow U/\FUgm$$ is a
pushout transmission in $\STROP_m$ covering $\gm$. \{Here we
identify $M/\FUgm =N$ in the obvious way.\}\end{thm}

In particular every initial transmission in  $\STROP_m$ is a
pushout transmission in  $\STROP_m$.

\begin{schol}\label{schol1.11}
We consider the special case that $U$ is a supertropical semiring.
If the supertropical monoid $U/\FUgm$ happens  to be  a semiring,
then it is  clear that $\pi_{\FUgm}$ is a pushout in the category
$\STROP$. Thus, following \cite[Notation 1.7]{IKR2}, we now have
$$\FUgm = \EUgm, \quad \pi_{\FUgm} = \al_{U,\gm}.$$
But if $U/\FUgm$ is not  a semiring then the relation $\FUgm$ is
different from $\EUgm$.
\end{schol}

If $U$  is a supertropical semiring and $\mfa$ is an ideal of $U$
we introduced in \cite[\S5]{IKR2} the saturum $\satUa$ and the
equivalence relation $\Ea = \EUa$, and obtained there descriptions
of these objects, which do not  mention addition but only employ
multiplication and the idempotent $e$ (\cite[Corollary 5.5,
Theorem 5.4]{IKR2}).

We now use these descriptions  to \emph{define} $\satUa$ and
$\EUa$ if $U$ is only a supertropical monoid.

\begin{defn}\label{defn1.12} Let $\mfa$ be  an ideal of the
supertropical monoid $U$.
\begin{enumerate} \ealph
    \item The \textbf{saturum} $\satUa$ of $\mfa$ is the set  of
    all $x \in U $ with $ex \leq  ea$ for some $a \in \mfa$. We call
    $\mfa$ \textbf{saturated} if $\satUa = \mfa$.

    \item The equivalence relation $E:= \Ea := \EUa$  is defined
    as follows:
    $$\begin{array}{lll}
    x \sim_E y & \Leftrightarrow & \text{either } x =y \\
&& \text{or } x \in \satUa, \  y \in \satUa.\\
  \end{array}
$$
\end{enumerate}
\end{defn}

As in \cite[\S5]{IKR2} the following fact can be verified in an
easy straightforward way.

\begin{prop}\label{prop1.13} $ $
\begin{enumerate} \eroman
    \item $\satUa$ is again a monoid ideal of $U$.
    \item The saturated ideals $\mfa$ correspond uniquely with the
    ideals $\mfc$ of $M$ which are lower sets of $M$, via
    $$ \mfc = \mfa \cap M = e \mfa \quad \text{and} \quad  \mfa = \{ x \in U | ex \in \mfc \}.$$
    \item $\EUa$ is a TE-relation on $U$.
    \item If $\mfb$ is a second  ideal of $U$ then
    $$ \EUa = \EUb \quad \text{iff} \quad \satUa = \satUb.$$
    \item $\piEa$ has the zero kernel $\satUa$.
    \item If $\al : U \to V$  is a transmission  then the zero
    kernel $\mfz_\al$ is a saturated ideal, and $\al$ factors
    through $\piEa$ iff $\mfa \subset \mfz_\al$.

\end{enumerate}

\end{prop}

As in \cite[\S5]{IKR2} we will use the alleviated  notation $x
\sima y $ for $x \simEa y$, $\pia$ for $\piEa$, and~$[x]_{\mfa}$
for the equivalence class $[x]_{E(\mfa)}$. \pSkip

It is immediate how to generalize  the definition of the
equivalence relation  $\EUaPh$ given in \cite[\S6]{IKR2} to the
case that $U$ is a supertropical monoid. We study only the case
where these relations are TE-relations, and we encode (without
loss of generality) the homomorphic equivalence relation $\Phi$ on
$M := eU$ by a homomorphism from $M$ to another semiring. All the
following can be verified  in a straightforward way.

\begin{thm}\label{thm1.14}
Let $U$ be a supertropical monoid and $\gm: M \to M'$  a
surjective homomorphism from $M := eU$ to a (bipotent) semiring~
$M'$. Further, let $\mfA$ be an ideal of $U$ containing~$M$ and
the saturated ideal
$$ \mfa_\gm := \{ x \in U \ | \ \gm(ex) = 0 \}.$$
\begin{enumerate} \eroman
    \item The equivalence relation  $\EUAgm$ on $U$, given by ($x_1,x_2 \in U$)
    $$\begin{array}{lll}
    x_1 \sim_E x_2 & \Leftrightarrow & \text{either } x_1 =x_2 \\
&& \text{or } x_1 \in \mfA, \ x_2 \in \mfA, \   \gm(ex_1) = \gm(ex_2),\\
  \end{array}
$$
is a TE-relation.

    \item The transmission $\pi_E : U \to U/E$ has the ghost
    kernel $\mfA$. The ghost part  $(\pi_E)^\nu$ is the map $\gm: M \to
    M'$. \{Here we identify the ghost ideal  $M/E$  of $U/E$ with $M'$ in the obvious way.\}
    \item Assume that a transmission $\bt: U \to V$ to a
    supertropical monoid $V$ is given with ghost kernel $\mfA_\bt \supset
    \mfA$, further a homomorphism  $\dl : M' \to eV$ is given such
    that $\dl \gm = \bt^\nu$. Then there exists a unique
    transmission  $\eta: U/ E \to V$ with $\eta^\nu = \dl$ and $\eta \al =
    \bt$.
\end{enumerate}
(cf. the diagram following \cite[Problem 1.1]{IKR2}.)
\end{thm}

\begin{rem}\label{rem1.15}
It can be readily verified that
$$E(U, M\cup \mfa_\gm, \gm) = F(U,\gm). $$
Thus the present theorem is a generalization of Theorem
\ref{thm1.10}.
\end{rem}

For any ideal $\mfA \supset M$ of $U$ we define
$$ E(U,\mfA) : = E(U,\mfA, \id_M),$$
as we did in  \cite[\S6]{IKR2} for $U$ a supertropical semiring.
In this special case Theorem \ref{thm1.14} reads as follows.

\begin{cor}\label{cor1.16} $E(U, \mfA)$  is a TE-relation on $U$.
A transmission $\al: U \to V$ (with $V$ a supertropical monoid)
factors through  $\pi_{E(U,\mfA)}$ iff $\mfA \subset \mfA_\al$.

\end{cor}

\section{Canonical factorization of a transmission}\label{sec:2}

Given a surjective transmission $\al : U \to V$ between
supertropical semirings we start out to write $\al$ as a
composition of transmissions of simple nature in a somewhat
canonical way. More precisely, we will do this first in the
category $\STROP_m$ of supertropical monoids. Afterward we will
prove that, if $U$ and $V$ are semirings, this ``canonical
factorization'' remains valid in the smaller category $\STROP$ of
supertropical semirings, which has our primary interest.

 Let $U$ and $V$ be supertropical monoids, and $M:=eU$, $N := eV$
 their ghost ideals. We first exhibit the ``transmissions of simple
 nature''  we have in mind. These are the \emph{ideal compressions,
 tangible fiber contractions}, and \emph{strict ghost
 contractions} to be defined now.
\begin{defn}\label{defn2.1} As in the case that $U$ and $V$ are
supertropical semirings (cf. \cite[\S6]{IKR1}) we say that a
surjective transmission $\al: U \to V$ is a \textbf{fiber
contraction} if the ghost part $\al^\nu = \gm: M \to N$ is an
isomorphism. We say that $\al$ is a \textbf{fiber contraction
over} $M$, if $N = M$ and $\gm  = \id_M$.
\end{defn}

Notice that $\al$ is a fiber contraction iff the equivalence
relation $\Eal$ is an MFCE-relation, hence $\al = \rho \circ
\pi_E$ with $E$ an MFCE-relation and $\rho$ an isomorphism. Then
$\al$ is a fiber contraction over $M$ iff $M=N$ and $\rho$ is an
isomorphism over $M$.

 \begin{defn}\label{defn2.2}
We call a surjective transmission $\al: U \to V$ an \textbf{ideal
compression}, if $\al$ is a fiber contraction over $M$ which maps
$U \setminus \mfA_\al$ bijectively on to $V \setminus N = \tT(M)$.
\{Recall that $\mfA_\al$ denotes the ghost kernel of $\al$.\}
\end{defn}

This means that $\al = \rho \circ \pi_{E(U, \mfA)}$ with $\mfA$ an
ideal of $U$ containing $M$ and $\rho$ an isomorphism from $\brU
:= U/ E(U,\mfA) $ onto $V$ over $M$. We have $\mfA = \mfA_\al$.

\begin{defn}\label{defn2.3} We call a transmission $\al$
\textbf{tangible} if
$$\al(\tT(U)) \ds  \subset \tT(V) \cup \{ 0 \},$$
and \textbf{strictly tangible} if
$$\al(\tT(U))  \ds \subset \tT(V).$$
\end{defn}
In other terms, $\al$ is tangible iff $\mfA_\al = M \cup
\mfz_\al$, and $\al$ is strictly tangible iff $\mfA_\al = M$.

What does this means in the case that $\al$ is a fiber
contraction? Clearly, a tangible fiber contraction $\al: U \to V$
is strictly tangible. If $E$ is an MFCE-relation on $U$, then
$\pi_E: U \to U/E$ is tangible iff $E$ is \textbf{ghost
separating} (cf. \cite[Definition 6.19]{IKR2}), in other terms,
iff $E$ is finer  than the equivalence relation $E_t := E_{t,U}$
on $U$ which has the equivalence classes $\{ a\in \tT(U) \ds | ex
= a \}$,  $a \in M \sm \{ 0\}$, and the one-point equivalence
classes $\{y\}$, $y \in M$ (cf. \cite[Example 6.4.v]{IKR1}).

\begin{defn}\label{defn2.4}
We call the MFCE-relations $E$ on $U$ with $E \subset E_t$
\textbf{tangible MFCE-relations.}
\end{defn}

In this terminology  the tangible fiber contractions $\al: U \to
V$ over $M := eU$ are the products
$$ \al = \rho \circ \pi_T$$
with $T$ a tangible MFCE-relation on $U$ and $\rho$ an isomorphism
over $M$.

\begin{defn}\label{defn2.5}
We call a transmission $\al: U \to V$ a \textbf{ghost
contraction}, if $\al^\nu$ is a homomorphism from $M$ onto $N$,
and if $\al$ maps $U \sm (M \cup \mfz_\al)$ bijectively onto
$\tT(V) = V \sm N$.
\end{defn}

This means that
$$ \al = \rho \circ \pi_{F(U,\gm)}$$
with $\gm: M \to N$ a surjective homomorphism, namely $\gm =
\al^\nu$, and $\rho$ an isomorphism over $N$ from $U / F(U,\gm)$
to $V$. Thus $\al$ is a ghost  contraction iff $\al$ is a
surjective pushout transmission in $\STROP_m$. \{The equivalence
relation $F(U,\gm)$ had been introduced in Theorem
\ref{thm1.10}.\}

\begin{defn}\label{defn2.6.0}
In the situation of Definition \ref{defn2.2} and Definition
\ref{defn2.5},  respectively,   we also say abusively that $V$ is
an ideal compression  (resp. a ghost contraction) of $U$.

\end{defn}

\begin{defn}\label{defn2.6}
We call a ghost contraction $\al : U \to V$ \textbf{strict}, if
$\al^{-1}(0) \subset M$. This means that $\al$ is also a strict
tangible transmission.
\end{defn}

Notice that every ghost contraction $\al: U \to V$ with
$\gm^{-1}(0) = \{ 0 \}$, $\gm = \al^\nu$, is strict, and that
$\gm^{-1}(0) = \{ 0 \}$ iff $\mfz_\al = \{ 0\}$. Of course, there
exist other strict ghost contractions. The maps $\pi_{F(U,\gm)},$
where $\gm: M \to N$ is a homomorphism with $\igm(0) \neq \00,$
but where $U$ has no tangibles with ghost companion in $\igm(0),$
are main examples for this.

\begin{defn}\label{defn2.7} If $\gm: M \to N$ is a surjective
homomorphism for $M = eU$ to a semiring $N$, we put
$$ \mfa_{U,\gm}:= \{ x \in U \ | \ \gm(ex) =0\},$$ an ideal already used in Theorem
\ref{thm1.14}.
\end{defn}
In this notation the ghost contraction $\pi_{F(U,\gm)}$ is strict
iff $\mfa_{U,\gm} \subset M$.

\begin{thm}\label{thm2.8}
Let $\al:U \to V $ be a surjective transmission between
supertropical monoids and let $\gm: M \to N$ denote the
homomorphism between the ghost ideals $M := eU$, $N:= eV$ obtained
from $\al$ by restriction, $\gm = \al^\nu$.
\begin{enumerate} \eroman
    \item There exists a factorization
    $$ \al = \mu \circ \bt \circ \lm$$ with $\lm$ and ideal
    compression of $U$, $\bt$ a strict ghost contraction, and
    $\mu$ a tangible fiber contraction over $N := eV$.

    \item The factorization is essentially unique. More precisely,
    if  $ \al = \mu' \circ \bt' \circ \lm'$ is a second such
    factorization of $\al$, then there exist isomorphisms $\rho$
    over $M$ and $\sig$ over $N$ (of supertropical monoids) such
    that
    $$  \lm' = \rho \lm, \quad \mu' = \mu \sig^{-1}, \quad \bt' = \sig \bt \rho^{-1}.$$

    \item In particular we can choose
$$ \lm = \pi_{E(U,\mfA)}: U \longrightarrow \brU := U/E(U,\mfA)$$
with $\mfA := \mfA_\al$, the ghost kernel of $\al$,
$$ \bt = \pi_{\FbUgm}: \brU \longrightarrow W := \brU/\FbUgm$$
and $\mu: W \onto V$ the resulting tangible fiber contraction over
$N$ such that $\al = \mu \bt \lm$ (see proof below).
\end{enumerate}

\end{thm}

\begin{proof}
a) Let $\gm := \al^\nu : M \to N$, $\mfA:= \mfA_\al$, and $\brU:=
U/E(U,\mfA)$. Then $\al$ factors through $\lm := \pi_{E(U,\mfA)}$
in a unique way,
$$
\xymatrix{  \al: U   \ar @{>}[r]^{\lm}  & \brU \ar@{>}[r]^
{\bral} & V,  \\
}
$$
with $\bral$ a surjective transmission having trivial ghost
kernel. This is clear from \cite[Proposition 6.20]{IKR2}, adapted
to the category of supertropical monoids. \pSkip

b) We have $(\bral)^\nu = \gm$. Let $\bt = \pi_{\FbUgm}$. By
Theorem \ref{thm1.10} we know that $\bt$ is an initial
transmission in the category $\STROP_m$ (even a pushout). Thus we
have a unique transmission
$$ \mu: W := \brU/ \FbUgm \ \to \ V$$
such that $\bral = \mu \circ \bt$,  hence $\al = \mu  \circ \bt
\circ  \lm $. From $(\bral)^\nu = \gm = \mu ^\nu \circ \bt ^\nu$
and $\bt^\nu = \gm$ it follows that $\mu^\nu$ is the identity of
$N$.

Since $\bral$ has trivial ghost kernel and $\bt$ is surjective,
both $\bt$ and $\mu$ have trivial ghost kernels. We conclude that
$\bt$ is a strict ghost contraction and $\mu$ is a tangible fiber
contraction over~$N$. Parts  i) and iii) of the theorem are
proven.

\pSkip

c) Retaining the transmissions $\lm, \bt, \mu$ which we have
defined above, we turn to the claim  of uniqueness in part ii) of
the theorem. Let $ \al = \mu' \circ \bt' \circ \lm'$ another
factorization of $\al$ of the kind considered here. Both $\bt'$
and $\mu'$ have trivial ghost kernel. Thus the ideal compression~
$\lm'$ has the same ghost kernel $\mfA$ as $\al$. We conclude that
$$ \lm' = \rho \pi_{E(U, \mfA)} = \rho \lm$$ with some isomorphism
$\rho$ over $M$.

 From $ \al = (\mu'  \bt' \rho) \lm$ we then conclude that $\mu'  \bt'
 \rho = \bral$. Now $\bt' \rho$ is a strict ghost contraction
 covering $\gm$, since $\bt'$ is such a ghost contraction and
 $\rho$ covers $\id_M$. It follows that
 $$ \bt' \rho = \sig \pi_\FbUgm = \sig \bt$$ with some isomorphism
 $\sig$ over $N$, and hence $\bt' = \sig \bt \irho$. We finally
 obtain
$$\al = \mu'  \sig \bt \lm = \mu \bt
 \lm,$$  and then  $\mu' \sig = \mu$.
\end{proof}

\begin{thm}\label{thm2.9}
Let $\al:U \to V $ be a surjective transmission between
supertropical semirings, and assume that
$$
\xymatrix{  \al: U   \ar @{>}[r]^{\lm}  & U_1 \ar@{>}[r]^
{\bt} & V_1  \ar @{>}[r]^{\mu}  & V,  \\
}
$$
is a factorization of $\al$ as described in Theorem \ref{thm2.8}.i
(in the category $\STROP_m$). Then both $U_1$ and $V_1$ are
supertropical semirings, hence all three factors $\lm,\bt,\mu$ are
morphisms in~$\STROP$.
\end{thm}

\begin{proof}
$\lm$ and $\mu$ are surjective and $\lm^\nu = \id_M$, $\mu^\nu =
\id_N$. Moreover $\mu$ has trivial ghost kernel. Thus $V_1$ is a
semiring by Theorem \ref{thm1.5}.i, and $U_1$ is a semiring by
Theorem \ref{thm1.5}.ii.
\end{proof}

\begin{cor}\label{cor2.10}
Let $\al:U \to V $ be a surjective transmission between
supertropical semirings covering $\bral = \gm : M \to N $. Then
for the supertropical semiring $$  \brU := U / E(U, \mfA_\al)$$
the transmission $$\pi_\FbUgm : \brU \to \brU / \FbUgm $$ is
pushout in the category $\STROP$. In other terms (cf.
\cite[Notation 1.7]{IKR2})
$$ \FbUgm = \EbUgm.$$
\end{cor}

\begin{proof} Theorem \ref{thm2.9} tells us that $\brU / \FbUgm$
is a supertropical semiring. We know from \S\ref{sec:1} that
$\pi_\FbUgm$ is pushout in $\STROP_m$. A fortiori this
transmission is pushout in $\STROP$.
\end{proof}

\begin{defn}\label{defn2.11}
Let $\al:U \onto V $ be a surjective transmission covering $\alnu
= \gm : M \onto N $. We know by Theorem \ref{thm2.8} that there
exists a unique factorization
 \begin{equation}\renewcommand{\theequation}{$*$}\addtocounter{equation}{-1}\label{eq:str.1}
 \al = \rho \circ  \pi_T \circ \pi_\FbUgm  \circ \pi_{E(U,\mfA)}
\end{equation}
with $\mfA$ an ideal of $U$ containing $M \cup \mfz_\al$, $\brU :=
U/ E(U,\mfA)$, $T$ a tangible MFCE-relation on $W:= \brU /
\FbUgm$, and $\rho$ an isomorphism from $W/T$ to $V$ over $N$.
Here $\mfA,$ $T$, and hence $\rho$ are uniquely determined by
$\al$. We call $(*)$ the \textbf{canonical factorization} of
$\al$, and $\pi_\FbUgm$, $\pi_{E(U,\mfA)}$, $\pi_{T}$, $\rho$ the
\textbf{canonical factors} of $\al$.

If one of these maps is the identity  map, we feel justified to
omit it in the list of the canonical factors of $\al$.
\end{defn}

We discuss some simple cases of canonical factorizations.
\begin{schol}\label{schol2.12} (The case of $M=N$.) Assume that
$U$ and $V$ are supertropical monoids with $eU = eV = M$. Let
$\al:U \to V$ be a fiber contraction over $M$ with ghost kernel
$\mfA = \mfA_\al$.\begin{enumerate} \eroman
    \item $\al$ has the factorization $\al = \mu \circ \lm$ with
    $\lm = \pi_{E(U,\mfA)}$ and
    $$\mu:  \brU :=
U/ E(U,\mfA) \to V$$ a (strict) tangible fiber contraction over
$M$. This is clear from Theorem \ref{thm2.8} or directly from the
universal property of $\pi_{E(U,\mfA)}$, cf. Corollary
\ref{cor1.16}. Then $\mu = \rho \circ \pi_T$ with $T$ a tangible
MFCE-relation on $\brU$. Thus $\al$ has the canonical factors
$\pi_{E(U,\mfA)}$ , $\pi_T$, and $\rho$.

    \item If $U$ is a  semiring then  both $\brU$ and
    $V$ are semirings, as follows directly from Theorem~\ref{thm1.5}.ii.
\end{enumerate}

\end{schol}

\begin{examp}\label{exmp2.13} (Factorization of a ghost
contraction.) Assume  $\al:U \onto V $ is a ghost contraction
covering $\alnu = \gm : M \onto N $. Let $\mfa$ denote the zero
kernel of $\al$, $\mfa := \mfz_\al$.
\begin{enumerate} \eroman
    \item $\al$ has the factorization $\al = \bt \circ \lm$ with
    $ \lm = \pi_{E(U, M\cup \mfa)}$ and
    $$ \bt: \brU := U / E(U, M\cup \mfa) \ \to \ V$$ a strict ghost
    contraction. We have
    $$ \bt = \rho \circ \pi_{\FbUgm}$$
with $\rho$ an isomorphism from $\brU / \FbUgm$ to $V$ over $N$.
Thus $\al$ has the canonical factors $\pi_{E(U, M\cup \mfA) }$,
$\pi_\FbUgm$, and $\rho$.

    \item We further have the factorization $$\bt = \brbt \circ \pi_{E(\bmfa)}$$
    with $\bmfa:= \lm(\mfa)$ and
    $$ \brbt : \brU/E(\bmfa) = U/E(\mfa) \ \to \ V,$$
    which is a strict ghost contraction with zero kernel $\{0\}$.
    Notice that the ideal $\mfa =\mfz_\al$ is saturated in $U$ and $\bmfa$ is saturated
    in $\brU$.

\end{enumerate}
\end{examp}

\begin{examp}\label{exmp2.14} (The transmissions $\pi_{E(U,\mfA,\gm)}$.)

Let $U$ be a supertropical monoid and let  $\gm: M \to N$ be a
surjective homomorphism from  $M = eU$ to a semiring $N$. Further
let $\mfA$ be an ideal of $U$ containing $M \cup \mfa_{U,\gm}$.

\begin{enumerate} \eroman
    \item Then the transmission
 $$ \pi_{E(U,\mfA,\gm)}: U \to V : = U/{E(U,\mfA,\gm)}$$ (cf. Theorem
 \ref{thm1.14}) has the canonical factorization
$$ \pi_{E(U,\mfA,\gm)} = \pi_\FbUgm \circ \pi_{E(U,\mfA)}$$
with $\brU := U/E(U,\mfA) $. Indeed we have ${E(U,\mfA,\gm)}/
E(U,\mfA) = F(\brU, \gm)$, as has been stated in \cite[Theorem
6.22]{IKR2}.

    \item The ghost contractions $\al: U \to V$ covering $\gm$ are
    precisely the maps
    $$ \al = \rho \circ  \pi_{E(U,\mfA,\gm)}$$ with $\mfA = M \cup
    \mfa_{U,\gm}$ and $\rho$ an isomorphism over $N$, as is clear from
    the above and Example \ref{exmp2.13}.
\end{enumerate}
\end{examp}

\section{The canonical factors of a product of two basic transmissions}\label{sec:3}

\begin{defn}\label{defn3.1}
Let $\al : U \to V$ be a surjective transmission between
supertropical monoids, and let $\gm:= \al^\nu : M \onto N$  denote
the ghost part of $\al$. We call $\al$ a \textbf{basic
transmission}, if~ $\al$ is of one of the following 4 types.
\begin{description}
    \item[\textbf{\emph{Type 1}}] $\al = \pi_{\EUmfA}$ with $\mfA$ an
    ideal of $U$ containing $M$. \pSkip
    \item[\textbf{\emph{Type 2}}] $\al = \pi_{\FUgm}$  and $\al^{-1}(0) =
    \gm^{-1}(0)$. \pSkip
    \item[\textbf{\emph{Type 3}}] $\al = \pi_T$ with $T$ a tangible
    MFCE-relation on $U$. \pSkip
    \item[\textbf{\emph{Type 4}}] $\al = \rho$ with $\rho$ an isomorphism
    over $M$.
\end{description}
Thus in all cases except the second we have $M=N$ and $\gm =
\id_M$.

In short, the  basic transmissions are the factors occurring in
the canonical factorizations of transmissions (cf. Definition
\ref{defn2.11}).
\end{defn}

\begin{problem}\label{prob3.2} Given basic transmissions $\al: U \to
V$ of type $i$ and $\bt: V \to W$ of type $j$ with $1 \leq j \leq
i\leq 4$, find the canonical factorization of $\bt \al$
explicitly.
\end{problem}

It would be easy to find these canonical factorizations up to an
undetermined isomorphism~ $\rho$ as first factor (cf. Definition
\ref{defn2.11}) by running through parts a) and b) of the proof of
Theorem~\ref{thm2.9}. But we want a completely explicit
description of all factors. For this we will rely on realizations
of the quotient monoids $U/\EUmfA$, $U/\FUgm$, $U/T$ arising  up
in Definition~\ref{defn3.1}, such that the basic transmissions
$\pi_{\EUmfA}$, $\pi_{\FUgm}$, $\pi_T$ have a particulary well
amenable  appearance,

\begin{convs}\label{convs3.2}
Let $U$ be a supertropical monoid, $\mfA$ an ideal of $U$
containing $M := eU$, furthermore  $\gm:M \to N$ a surjective
homomorphism to a (bipotent)  semiring $N$ with $\mfa_{U,\gm}
\subset M$ (cf. Definition \ref{defn2.7}), and  $T$ a tangible
MFCE-relation on $U$.

\begin{enumerate} \ealph
    \item We write $\mfA := M \dcup S$ with $S$ a subset of
    $\tT(U)$ such that  $S \cdot \tT(U) \subset S \cup M$.
    Justified by \cite[Theorem 6.16]{IKR2}, adapted to the monoid
    setting, we declare that $U/\EUmfA$ is the subset
    $$ U \sm  S = (\tT(U) \sm S) \ds {\dcup} M$$
    of $U$, and, for any $x\in M$
$$ \pi_{\EUmfA} (x) = \left\{
\begin{array}{lll}
  x & \text{if} & x \in U \sm S, \\[1mm]
  ex & \text{if} & x \in S.\\
\end{array}
\right.$$
For $x,y \in U \sm S$ the product $x \odot y$ in $U/ \EUmfA$ is
given by
$$ x \odot y = \left\{
\begin{array}{rll}
  xy & \text{if} & xy \notin S, \\[1mm]
  exy & \text{if} & xy \in S.\\
\end{array}
\right.$$

\item We identify $\tT(U / \FUgm) $ with $\tT(U)$ such that $[x]_\FUgm =
x$ for $x \in \tT(U)$. Now
$$ \tT(U / \FUgm) = \tT(U) \ds{\dcup} N$$
and, for $x \in U$,
$$ \pi_{\FUgm} (x) = \left\{
\begin{array}{lll}
  x & \text{if} & x \in \tT(U), \\[1mm]
  \gm(x) & \text{if} & x \in M.\\
\end{array}
\right.$$
If $x,y \in \tT(U) $, the product $x \odot y$ in $U/ \FUgm$ is
given by
$$ x \odot y = \left\{
\begin{array}{lll}
  xy & \text{if} & xy \in \tT(U), \\[1mm]
  \gm(xy) & \text{if} & xy \in M.\\
\end{array}
\right.$$

\item For $x \in M $ we identify $x$ with $[x]_T$ (as we usually did for
MFCE-relations before, but notice that now $[x]_T = \{ x\}$). We
have
$$ U / T  = \tT(U)/ T \ds{\dcup} M, $$
and, for $x \in U$,
$$ \pi_{T} (x) = \left\{
\begin{array}{lll}
  [x]_T & \text{if} & x \in \tT(U), \\[1mm]
  x & \text{if} & x \in M.\\
\end{array}
\right.$$
If $x,y \in \tT(U) $ the product of $[x]_T $ and $[y]_T$ in $U/T$
is given by
$$ [x]_T \odot [y]_T = \left\{
\begin{array}{lll}
  [xy]_T & \text{if} & xy \in \tT(U), \\[1mm]
  xy & \text{if} & xy \in M.\\
\end{array}
\right.$$

\end{enumerate}

\end{convs}

We also need more terminology on equivalence relations.

\begin{defn}\label{defn3.4} $ $
\begin{enumerate} \eroman
    \item If $\eta : X \to Y$ is a map between sets and $F$ is an
     equivalence relation   on $Y$, then~ $\eta^{-1}(F)$ denotes the
     equivalence relation on $X$ given by
     $$ x_1 \sim_{\eta^{-1}(F)} x_2  \dss{\iff} \eta(x_1) \sim_F \eta(x_2).$$
Thus $\pi_{\eta^{-1}(F)} = \pi_F \circ \eta$.

\item We further have a unique map
$$
\xymatrix{ \breta: X / \eta^{-1}(F) \ar @{>}[r]   & Y / F  }
$$
such that the diagram
$$
\xymatrix{   X   \ar @{>}[d]^{\eta}  \ar
@{>}[rr]^{\pi_{\eta^{-1}(F)} } &   & X / \eta^{-1}(F)
\ar @{>}[d]^{\breta}   \\
Y  \ar @{>}[rr]^{\pi_F}  &  & Y / F  }
$$
commutes. We denote this map $\breta$ by $\eta^F$, and then have
the formula
$$ \pi_F \circ \eta \ds = \eta^F \circ \pi_{\eta^{-1}(F)}.$$

\item If $E$ is an equivalence relation on the set $X$ and $F$ is
an equivalence relation on $X / E$, then let $F \circ E$  denote
the equivalence relation on $X$ given by
$$  x_1 \sim _{F \circ E} x_2  \dss \iff x_1 \sim_E x_2 \ \text{ and } \ [x_1]_E \sim_F [x_2]_E.$$
We identify the sets $X / F \circ E$ and $(X / E) / F$ in the
obvious way. Then
$$ \pi_{F \circ E} = \pi _ F \circ \pi_E.$$
\end{enumerate}

\end{defn}
The following fact is easily verified.

\begin{lem}\label{lem3.5} Let $\eta: V \to U$ be a transmission
between supertropical monoids, and let $E$ be a TE-relation on
$U$. Then $\ieta(E)$ is a TE-relation on $V$, and the induced map
$$ \eta^E : V/ \ieta(E) \to U/E$$
is again a transmission. We have the formula
$$ \pi_E \circ \eta = \eta^E \circ \pi_{\ieta(E)}.$$
\end{lem} \pSkip

This lemma already gives us the solution of Problem \ref{prob3.2}
for $\al$ basic of type 4.

\begin{prop}\label{prop3.6} Let $U$ and $V$ be supertropical
monoids with $eU = eV = :M$, and let $\rho: V \to U$ be a tangible
fiber contraction over $M$ (e.g. $\rho$ is an isomorphism over
$M$).
\begin{enumerate} \ealph
    \item If $T$ is a tangible MFCE-relation on $U$, then
    $\irho(T)$ is a tangible MFCE-relation on $V$ and
    $$ \pi_T \circ \rho = \rho' \circ \pi_{\irho(T)}$$
    with $\rho' : V/ \irho(T) \to U/T$ the obvious homomorphism
    over $M$ induced by $\rho$, namely $\rho' = \rho^T$.

    \item If $\gm: M \onto N$ is a surjective homomorphism from
    $M$ to a semiring $N$ with
    $$ \mfa_{V, \gm} := \{ x \in V \ds |  \gm(ex) = 0\} \subset M, $$
    hence also $\mfa_{V, \gm} \subset M$, then
    $$ \pi_{\FUgm} \circ \rho = \rho' \circ \pi_{\FUgm}$$
with
    $$ \rho' := \rho^{\FUgm} : V/\FVgm \ds \to U/ {\FUgm}.$$
$\rho'$ is a tangible fiber contraction over $N$ with the
following explicit description: \\  Writing   $ V/\FUgm  = \tT(V)
\dcup N$ and $ U/\FUgm  =  \tT(U) \dcup N$ (cf.
Convention~\ref{convs3.2}.b), we have $\rho'(x) = \rho(x)$ if $x
\in \tT(U)$ and $\rho'(x) = x $ if $x \in N$.

\item Let $\mfA$ be an ideal of $U$ containing $M$ and let $\mfB :=
\irho(\mfA)$, which is an ideal of $V$ containing $M$. Then
    $$ \pi_{\EUmfA} \circ \rho = \rho' \circ \pi_{\EVmfB},$$
with
    $$ \rho' := \rho^{\EVmfB} : V/\EVmfB \ds \to U/ {\EUmfA}.$$
$\rho'$ is a tangible fiber contraction over $N$, which has the
following explicit description:  We write $ \mfA  = M \dcup S$
with $S \subset \tT(U)$, and have $ \mfB  =  M \dcup \irho(S)$
with $\irho(S) \subset \tT(V)$. By Convention   \ref{convs3.2}.a
    $$ U/ {\EUmfA} = (\tT(U) \sm S) \ds{\dcup} M,$$
    $$ V / {\EVmfB} = (\tT(V) \sm \irho(S)) \ds{\dcup} M.$$
The map $\rho'$ is obtained from $\rho$ by restriction to these
subsets of $U$ and $V$. \{N.B. It is easy to check directly that
$\rho'$ respects multiplication.\}

\item If $\rho$ is an isomorphism, then in all three cases $\rho'$ is again an
isomorphism. Thus, if $\bt:U \to W$ is a basic transmission of
type $i = 1,2,3$ and $\rho: V \to U$  is basic of type~4, then
$\bt \rho = \rho' \bt'$ with $\bt', \rho'$ again of type $i$ and 4
respectively.
\end{enumerate}

\end{prop}
\begin{proof}
Straightforward by  use of Lemma \ref{lem3.5}.
\end{proof}

\begin{rem}\label{rem3.7} If in Proposition \ref{prop3.6}.b we dismiss the assumption that $\mfa_{V, \gm} \subset
M$, we have the same result, but with a slightly more complicated
description of the tangible fiber contraction~$\rho'$ as follows:
We now have natural identifications
$$ V / \FVgm = ( \tTV \sm \mfa_{V, \gm}) \ds \cup N,$$
$$ U / \FUgm = ( \tTU \sm \mfa_{U, \gm}) \ds \cup N,$$ and then
$$ \rho'(x) = \left\{
\begin{array}{lll}
  \rho(x) & \text{if} &  x\in \tTV \sm  \mfa_{V, \gm} \\[1mm]
  x & \text{if} &  x\in N. \\
\end{array}
\right.$$
\end{rem}

The following three propositions contain the solution of Problem
\ref{prob3.2} in the remaining cases $i \leq j \leq 3$. The stated
canonical factorizations can always quickly be verified by
inserting an element $x$ of $\tTU$ and comparing both sides. \{For
$x \in eU$ equality is always evident.\} Often more conceptional
proofs are also possible. With one exception we do not give the
details.

\begin{prop}\label{prop3.8} (The case $i=j$.)

 Let $U$ be a
supertropical monoid and $M := eU$.
\begin{enumerate} \ealph
    \item Assume that $\mfA = M \dcup S$ is an ideal of $U$
    containing $M$ and $\bmfB = M \dcup S'$ is an ideal of $\brU := U / \EUmfA = U
    \sm
    S$ containing $M$. \{Thus $S$ and $S'$ are disjoint subsets of
    $\tTU$.\} Then
    $$ \mfB := \pi^{-1}_{\EUmfA}(\bmfB) = M \dcup S \dcup S'$$
    is an ideal of $U$ and
    $$ \pi_{\bEUmfB} \circ \pi _{\EUmfA} = \pi_{\EUmfB}.$$
    \item Let $\gm: M \to N$ and $\dl: N \to L$ be surjective
    homomorphisms of bipotent semirings with $\mfa_{U,\gm} \subset M$ and $\mfa_{V,\dl} \subset N$,
     where $V := U / \FUgm$. Then
    $$ \pi_{F(V,\dl)} \circ \pi _{\FUgm} = \pi_{F(U, \dl \gm)}.$$

    \item If $T$ is a tangible MFCE-relation on $U$ and $T'$ is a
    tangible MFCE-relation on $U/ T,$ then $T' \circ T$ is again a
    tangible MFCE-relation and
    $$ \pi_{T'} \circ \pi_T = \pi_{T' \circ T}.$$
\end{enumerate}

\end{prop}

\begin{prop}\label{prop3.9} (The case $\al = \pi_T$.)

Let $U$ be a supertropical monoid, $T$ a tangible MFCE-relation on
$U$, and $$\brU := U/T = (\tTU/T) \dcup M.$$
\begin{enumerate} \ealph
    \item If $\gm$ is a surjective homomorphism from $M:= eU$ to a
    semiring $N$, and $\mfa_{U,\gm} \subset M$, then
    $$ \pi _{F(\brU, \gm)} \circ \pi_T = \pi_{T'} \circ \pi_{\FUgm},$$
where $T'$ is the tangible MFCE-relation on $V:= U/ \FUgm = \tTU
\cup N$ defined as follows. For any $x,y \in V$
$$ x \sim_{T'}y  \quad \Leftrightarrow \quad \left\{
\begin{array}{lll}
  \text{either} &  &  x,y \in \tTU
  \ \text{ and } \  x \sim_T y, \\
  \text{or} &   & x = y \in N. \\
\end{array}%
\right.$$

    \item Let $\mfB$ be an ideal of $\brU$ containing $M$, hence
    $\mfB = M \dcup \brS$ with $\brS$ a subset of $\tTU / T$. We put $S:= \ipi
    _T(\brS) \subset \tTU$.
    Then
    $$ \pi_{E(\brU, \mfB)} \circ \pi_T = \pi_{T'} \circ \pi_{E(U, \mfA)}, $$
    with $\mfA := \ipi_T(\mfB) = M \cup S$, and $T'$ a tangible
    MFCE-relation on $$V:= U/ \EUmfA = (\tTU \sm S) \dcup M.$$
    $T'$ is obtained from $T$ by restriction to the subset $U \sm
    S$ of $U$. \{Notice that $\tT(U) \sm S$ is a union of $T$-equivalence classes.\}
\end{enumerate}
\end{prop}

\begin{prop}\label{prop3.10} (The remaining case $j=1, i=2$.)

 Let $U$ be a
supertropical monoid and $\gm$ a homomorphism from $M := eU$ onto
a semiring~$N$ with $\mfa_{U,\gm} \subset M$. Let $V:= U/ \FUgm =
\tTU \dcup N$.

\begin{enumerate} \eroman
    \item The ideals $\mfA \supset M $ of $U$ correspond uniquely
    with the ideals $\mfB \supset N $ of $V$ via $\mfA =
    \ibt(\mfB),$ $\mfB = \bt(\mfA)$, where $\bt:= \pi_\FUgm$. We
    then have $\mfA = M \dcup A$, $\mfB = N \dcup S$ with the same
    set $S \subset \tTU = \tTV$, and $\EUmfA = \ibt(\EVmfB)$.
    Finally
    \item

    $$ \pi_{\EVmfB} \circ \pi_{\FUgm} = \pi_{F(\brU,\gm)} \circ \pi_{\EUmfA}$$
    with $\brU := U/ \EUmfA = (\tTU \sm S) \dcup M$.
\end{enumerate}
\end{prop}

\begin{proof} i): The point is that for $S$ a subset of $\tTU$ we
have $S \cdot \tTU \subset M$ in $U$ iff   $S \cdot \tTU \subset
N$ in $V$, since $\ibt(N) = M$. (Recall that we identified $\tT(U)
= \tT(V)$.)

ii): Again just insert a given $x \in \tTU$ in both sides of the
equation and compare.
\end{proof}

\begin{summr}\label{prop3.11}
If $\al:U \to V$ and $\bt: V \to W$ are basic transmissions, $\al$
of type $i$ and $\bt$ of type $j \leq i$, cf. Definition
\ref{defn3.1}, then in case $i = j$ the transmission $\bt \al$  is
again basic of type $i$, and otherwise $\bt \al = \al' \bt'$
with~$\al'$ basic of type $i$ and $\bt'$ basic of type $j$, and
the new basic transmissions can be determined from $\al$ and $\bt$
in an explicit way.
\end{summr}
If $\al:U \to V$ and $\bt: V \to W$ are any transmissions with
known canonical factors, then the canonical factorization of $\bt
\al$ can be determined explicitly in at most $4+3+2+1 = 10$ steps.

\section{The semiring associated to a supertropical monoid; \\ initial transmissions}\label{sec:4}

Let $U$ be a supertropical monoid and $M:= eU$ its ghost ideal. We
start out to convert $U$ into a supertropical semiring in a
somewhat canonical way.

If $S$ is any subset of $U$, then the set $(US) \cup M$ is the
smallest ideal of $U$ containing both~$S$ and $M$. For convenience
we introduce the notation
$$ \EUS \ds{:=} E(U, US \cup M).$$
Corollary \ref{cor1.16} tells us the meaning of this equivalence
relation.
\begin{schol}\label{schol4.1}
A transmission $\al: U \to V$ factors through the ideal
compression $\pi_{\EUS}$ (in a unique way), iff the ghost kernel
$\mfA_\al$ contains the set $S$.
\end{schol}

We now  define a subset $S(U)$ of $\tTU$, for which the relation
$E(U, S(U))$ will play a central role for most of the rest of the
paper.
\begin{defn}\label{defn4.2} $ $
\begin{enumerate} \ealph
    \item We call an element $x$ of $U$ an \textbf{NC-product} (in
    $U$), if there exist elements $y,z$ of $U$ and $y'$ of $M$ with
$$ x = yz, \quad y' < ey, \quad y' z = eyz.$$
Here the label ``NC'' alludes  to the fact that we meet a
non-cancellation situation in the monoid $M$: We have $y' \neq
ey$, but $y' z = eyz$.

    \item We denote the set of all NC-products in $U$ by $D_0(U)$ and
    the set $D_0(U) \cup M$ by $D(U)$. We finally put
    $$ S(U):= D(U) \sm M = D_0(U) \cap \tTU.$$
    This is the set of tangible NC-products in $U$.
\end{enumerate}
\end{defn}

Clearly $D_0(U) \cdot U \subset D_0(U)$. Thus $D(U)$ is an ideal
of $U$ containing $M$. We have
$$ E(U,S(U)) = E(U,D_0(U)) = E(U,D(U)).$$

Theorem~\ref{thm1.2} tell us that $U$ is a semiring iff $S(U) =
\emptyset$, i.e., $D(U) = M$.

We compare the set $S(U)$ with $S(V)$ for $V$ an ideal compression
of $U$.

\begin{lem}\label{lem4.3} Let $\mfA$ be an ideal of $U$ containing
$M$, $\mfA = M \dcup S$ with $S \subset \tTU$. We regard $V:= U/
\EUmfA$ as a subset of $U$, as explained in Convention
\ref{convs3.2}.a. Then
$$ S(V) = S(U) \sm S.$$
\end{lem}
\begin{proof}
It is obvious from  the description of $V$ in Convention
\ref{convs3.2}.a.  that $S(V) = S(U) \cap \tTV$, and we have $\tTV
= \tTU \sm S$.
\end{proof}

\begin{lem}\label{lem4.4}
If $T$ is a tangible MFCE-relation then
$$ S(U/T) = S(U)/T.$$
\end{lem}
\begin{proof}
Look at the description of $U/T$ in Convention \ref{convs3.2}.c.
\end{proof}

\begin{thm}\label{thm4.5} $ $
\begin{enumerate} \eroman
    \item The supertropical monoid $\htU:= U/ \EUSU $ is a
    semiring.
    \item The ideal compression
    $$\sig_U := \pi_{\EUSU}: U \ds{\to} \htU $$ is universal among
    all fiber contractions $\al: U \ds{\to} V$ with $V$ a semiring.
    More  precisely, given such a fiber contraction   $\al$, we
    have a (unique) fiber contraction \\ $\bt: \htU \to V$ with $\al = \bt \circ
    \sig_U$. \{N.B. If $\al$ is a fiber contraction over $M$, the same holds for $\bt$.\}
\end{enumerate}

\end{thm}
\begin{proof} (i): By Lemma \ref{lem4.3}, the set $S(\htU)$ is
empty; hence $\htU$ is a semiring. \pSkip

(ii): We may assume that $\al: U \to V$ is a fiber contraction
over $M$, and then that  $\al = \pi_T \circ \pi_{\EUmfA}$ with an
ideal $\mfA \supset M$ of $U$ and $T$ a tangible equivalence
relation on $\brU := U / \EUmfA$. By Lemma~\ref{lem4.4} the set
$S(V)$ is empty iff $S(\brU)$ is empty, and by Lemma~\ref{lem4.3}
this happens iff $S(U) \subset \mfA$. Then $\pi_{\EUmfA}$, and
hence $\al$, factors through  $\pi_{\EUSU} = \sig_U$  (cf.
Scholium \ref{schol4.1}). Conversely, if $\al = \bt \circ \sig_U$
with $\bt: \htU \to V$ a fiber contraction, then we  know already
by Theorem \ref{thm1.5}.ii that $V$ is a semiring, since $\htU$ is
a semiring.
\end{proof}

We call $\htU$ the \textbf{the semiring associated to the
supertropical monoid} $U$.

\begin{thm}\label{thm4.6} Assume that $U$ is a supertropical
semiring and $\gm$ is a surjective homomorphism from  $M := e U$
to a (bipotent) semiring $N$. Let $V:= U / \FUgm$, which may be
only a supertropical monoid. Then
$$ \al:= \sig_V \circ \pi _\FUgm : U \onto V \onto \htV$$
(with $\htV$ and $\sig_V$  as defined in the preceding theorem) is
the initial transmission from $U$ to a supertropical semiring
covering $\gm$ (cf. \cite[Definition 1.3]{IKR2}). In the Notation
1.7 of \cite{IKR2} this reads
$$ \sig_V \circ \pi_{\FUgm} = \al_{U,\gm}.$$
\end{thm}
\begin{proof}
Let $\bt: U  \onto W$ be a transmission to a supertropical
semiring $W$ covering $\gm$ (in particular, $eW = N$). Since
$\pi_\FUgm$ is an initial  transmission in the category $\STROP_m$
covering $\gm$, we have a (unique) transmission $\eta: V \to W$
over~$N$, hence fiber contraction over~$N$, with $\bt= \eta \circ
\pi_\FUgm$.

Theorem \ref{thm4.5} gives us a factorization $\eta = \vrp \circ
\sig_V$ with $\vrp : \htV \to W$ again a fiber contraction over
$N$. Then
$$ \bt = \vrp \circ \sig_V \circ \eta = \vrp \circ \al$$
is the desired factorization of $\bt$ in the category $\STROP$. Of
course, the factor $\vrp$ is unique,  since $\al$ surjective.
\end{proof}

We want to find the canonical factorization of $\al_{U,\gm}$. More
generally we look for the canonical factors of
$$ \al:=  \pi_\EVmfB \circ \pi_{\FUgm}$$
with $V:= U/ \FUgm$ and $\mfB$ an ideal of $V$ containing $N =
eV$. We allow $U$ to be any supertropical monoid.

We write $\mfB = N \cup S$ with $S \subset \tT(V)$. Similarly  to
Convention \ref{convs3.2}.b (which treats a special case) we have
a natural identification
$$ \tTV = \tTU \sm \mfa_{U,\gm}$$
in such a way that for every $x\in U$
$$ \pi_\FUgm(x) = \left\{
\begin{array}{lll}
  x & \text{if } \  x\in \tTU \sm  \mfa_{U, \gm}, \\[1mm]
  \gm(ex) & \text{otherwise} &  .\\
\end{array}
\right.$$ We then obtain the following generalization of
Proposition \ref{prop3.10}, arguing essentially in the same way as
in \S\ref{sec:3}.
\begin{lem}\label{lem4.7}
Let $V:= U / \FUgm$ and $\bt:= \pi_{\FUgm}$.
\begin{enumerate} \eroman
    \item The ideals $\mfB$ of $V$ containing $N = eV$  correspond
    uniquely with the ideals $\mfA$ of $U$ containing $M \cup
    \mfa_{U,\gm}$ via  $\mfA = \ibt(\mfB)$, $\mfB = \bt (\mfA)$. Writing $\mfB = N \dcup S'$, with
    $$ S' \subset \tTV = \tTU \sm \mfa_{U,\gm},$$
    we have $\mfA = M \dcup S$ with
    $$ S := \{ x\in \tTU \ds | \gm(ex) = 0 \} \cup S' \subset \tTU.$$
    \item $\pi_\EVmfB \circ \pi_\FUgm = \pi_{F(\brU, \gm)} \circ \pi_\EUmfA
    $, with $\brU := U / \EUmfA = (\tT(U) \sm S ) \dot \cup M.$
\end{enumerate}
\endbox
\end{lem}

In the case $\mfB = D(V)$ we have $S' = S(V)$. Thus the elements
of $S'$ are the products  $yz \in \tTV \subset \tTU$ with $\gm(y')
< \gm (ey)$ and $\gm(y'z) = \gm (eyz)$ for some $y' \in M$. Notice
that this forces $\gm(y') \neq 0$.

\begin{defn}\label{defn4.8} Let $U$ be any supertropical monoid. We
call an element $x$ of $U$ a \\ \textbf{$\gm$-NC-product} (in
$U$), if there exist elements $y' \in M$, $y \in U$, $z \in U$
with $x = yz$ and $\gm(y') < \gm (ey)$, $\gm(y'z) = \gm (eyz)$. We
denote the set of these elements $x$ by $D_0(U,\gm)$ and the set
$D_0(U,\gm) \cap \tTU$ of tangible $\gm$-NC-products by
$S(U,\gm)$.

Notice that $D_0(U,\gm)$ is an ideal of $U$. We further define
$$ D(U,\gm) :=  M \cup D_0(U,\gm) = M \cup S(U,\gm),$$
which is an ideal of $U$ containing $M$.
\end{defn}

In this terminology we have $S' = S(U,\gm)$. If $U$ is a semiring,
then we read off from Theorem \ref{thm4.6} and Lemma \ref{lem4.7}
the following fact.

\begin{thm}\label{thm4.9}
Let $U$ be a semiring and $\gm: eU =M \to N$ a surjective
homomorphism from~$M$ to a semiring $N$. Then $\al_{U,\gm}$ has
the canonical factorization $$ \al_{U,\gm} = \pi_{F(\brU, \gm)}
\circ \pi_{\EUS}$$ with $$S= \{ x \in \tTU \ds | \gm(ex) = 0\}
\cup S(U,\gm) $$ and $\brU = U/\EUS$.
\end{thm}

It is now easy to write down the equivalence relation
$E(\al_{U,\gm}) = E(U,\gm)$ (cf. Notation 1.7 in \cite{IKR2}). We
obtain
\begin{cor}\label{cor4.10}
For $U$ and $\gm$ as above, the equivalence relation  $\EUgm$
reads as follows ($x_1,x_2 \in U$):
$$ x_1 \sim_{\EUgm} x_2  \quad \Leftrightarrow \quad \left\{
\begin{array}{lll}
  \text{either} &  &  x_1 = x_2,\\[1mm]
  \text{or} &  &  x_1, x_2 \in D(U,\gm), \  \gm(e x_1) = \gm (ex_2), \\[1mm]
  \text{or} &   &  \gm(e x_1) = \gm (ex_2) = 0. \\
\end{array}%
\right.$$
\endbox
\end{cor}

If $N$ is cancellative then $S(U,\gm) = \emptyset$, and we fall
back to the description of $\EUgm$ in \cite[Theorem 1.11]{IKR2}.

Our arguments leading to Theorems \ref{thm4.6} and \ref{thm4.9}
make sense if we only assume  that $U$ is a supertropical monoid.
To spell this out we introduce an extension of Notation 1.7 in
\cite{IKR2}.
\begin{defn}\label{defn4.11} Let $U$ be a supertropical monoid
with ghost ideal $M:= eU$, and let \\ $\gm: M \to N$ be a
surjective semiring homomorphism.
\begin{enumerate} \eroman
    \item We \textbf{define} $U_\gm = \htV$ with $V:= U/ \FUgm$.
    Thus $U_\gm$ is a supertropical semiring.
    \item We \textbf{define}
    $$ \al_{U,\gm} := \sig_V \circ \pi_\FUgm : U \to U_\gm.$$
    \item We finally \textbf{define} $\EUgm := E(\al_{U,\gm})$ and
    then have $U_\gm = U/\EUgm$.
\end{enumerate}
\end{defn}
The arguments leading  to Theorems \ref{thm4.6} and \ref{thm4.9}
give more generally the following

\begin{thm}\label{thm4.12} $ $
\begin{enumerate} \eroman
    \item Given a transmission $\bt: U \to W$ from a supertropical
    monoid $U$ to a supertropical semiring  $W$ covering $\gm$ (in particular $eW=
    N$), there exists a unique semiring homomorphism $\eta: U_\gm \to W$
    over $N$ such that $\bt = \eta \circ \al_{U,\gm}$.
    \item $\al_{U,\gm}$ has the same canonical factorization as
    given in Theorem \ref{thm4.9} for $U$ a semiring, and  $\EUgm$ has
    the description written down  in Corollary \ref{cor4.10}.
\end{enumerate}
\endbox
\end{thm}

Given a further semiring homomorphism $\dl: N \to L$ we may ask
whether there exists  a transmission $\eta: U_\gm \to U_{\dl \gm}$
covering $\dl$. In other words, is $E(U, \dl \gm) \supset
E(U,\gm)$?

In general the answer will be negative. Assume for simplicity that
$\mfa_{U,\gm} = \mfa_{U,\dl \gm} = \{0 \}$ (or even, that $\igm(0)
= \{ 0\}$, $\idl(0) = \{ 0\}$). We have to study the commutative
diagram
 $$\xymatrix{
     & & U_\gm      \ar@{-->}[rr]^?_{\eta}  & &  U_{\dl \gm}\\
           U     \ar[rr]   \ar[rru]^{\al_{U,\gm}}       &  & U/\FUgm  \ar[rr]  \ar[u]    &&
           U/ F(U, \dl \gm) \ar[u]
       \\
   M  \ar[rr]_\gm \ar[u]  &&  N    \ar[rr]_\dl \ar[u]  &&       L \ar[u]
 }$$
where the unadorned arrows are the obvious natural maps.   Assume
further that $L$ is cancellative. Using Convention
\ref{convs3.2}.b we have
$$ \tTU = \tT(U/\FUgm) = \tT(U/ F(U, \dl \gm)) = \tT(U_{\dl \gm}),$$
but $\tT(U_\gm) = \tTU \sm S(U,\gm)$. If $\eta$ would exist then
$\eta \circ \al_{U,\gm}$ would restrict to the identity on $\tTU$.
But this cannot happen as soon as $S(U,\gm)$ is not empty. In
particular we realize the following:

\begin{rem}\label{rmk:4.13} If $U$ is a semiring, $\igm(0) = \{0
\}$, but $S(U,\gm) \neq \emptyset$, and if there exist a
homomorphism $\dl: N \onto L$ of semirings with $L$ cancellative
and $\idl(0) = \{0 \}$, then the initial transmission
$\al_{U,\gm}$ in $\STROP$ is \textbf{not} a pushout transmission.

\end{rem}

It is not difficult to find cases where the situation described
here is met.
\begin{examp}\label{exmp4.14}
$ $
\begin{enumerate} \ealph
    \item We choose a totally ordered abelian group  $G$ and  a
    convex subgroup $H$ of  $G$ with $H \neq \{ 1 \}$, $H \neq G $. The
    group $G /H $ is again totally ordered in a  unique way such
    that the map $G \to G /H$, $g \mapsto g H$, is order
    preserving.  Thus we have bipotent semifields  $\tlM := G \cup
    \00$ and $\tlL := (G/H) \cup \00$ at hand. Let $$ A:= \{ a \in G \ds | a < H \}
    \dss{\text{and}}
    A/ H := \{ aH \ds |  a \in A\}.    $$
    Then
    $$ M := A \cup H \cup \00, \qquad L:= (A/H) \cup \{ 1 \cdot H\} \cup \00
    $$ are subsemirings of $\tlM$ and $\tlL$, and hence are
    cancellative bipotent semidomains.

    \item We construct a noncancellative bipotent semiring $N$ as
    follows. As an ordered set, we put
    $$N := (A/ H ) \ds {\dot \cup}  H \ds{\dot \cup}  \00$$
    with $0 < A/H < H$, keeping the given orderings on $A/H$ and $H$. We
    decree that the multiplication on $N$ extends the given
    multiplication on $A/H$ and $H$, and, of course, $0 \cdot x = x \cdot 0 = 0
    $ for all $x\in N$, while $(aH) \cdot h := a H$ for  $a \in
    A$, $h \in H$.  This multiplication clearly  is associative
    and
    commutative, has the unit element $1 \in H$, and is
    compatible with the ordering on $N$. Thus $N$ can be
    interpreted as a supertropical semiring.

    \item We define maps $\gm : M \to N$ and $\dl : N \to L$  by
    putting $\gm(0) : = 0 $ and $\dl(0) : = 0 $, $\gm(a) := a H$, $\gm(h) :=
    h$, $\dl(aH) := a H$, $\dl (h) := 1 \cdot H = 1_L$ for $a \in
    A$, $h \in H$. Clearly~$\gm$ and $\dl$ are order preserving
    surjective monoid homomorphisms, hence are surjective semiring homomorphisms. We have $\igm(0) = \00$,
    $\idl(0) = \00$.

  \item We choose a homomorphism $\tlv: \tlT \to G$ from an
  abelian group $\tlT$ onto $G$. Then, by \cite[Construction 3.16]{IKR1}, we have a supertropical semifield
  $$ \tlU := \STR(\tlT, G, \tlv)$$
  at hand with $\tT(U) = \tlT$, $\tG(\tlU) = G$, $ex = \tlv(x)$
  for $x \in \tlT$. Let $$T:= \tlv^{-1}(A \cup H),$$ and let $v: T\onto A\cup
  H$ denote the monoid homomorphism obtained from $\tlv$ by
  restriction. The subsemiring
  $$ U := \STR(T, A\cup H, v)$$ of $\tlU$ is a supertropical
  domain with ghost ideal $M$ and $\tT(U) = T.$

    \item We take elements $h_1 < h_2$ in $H$  and $a \in A$. Then
    we take elements $x_1, x_2, y \in T$ with $v(x_1) = h_1$, $v(x_2) =
    h_2$, $v(y) = a$. Now we have $$\gm(ex_1) = h_1 < \gm(ex_2) =
    h_2 \dss{\text{and}} \gm(ex_1y) = \gm(e x_2 y ) = aH.$$ Also
    $x_2 y \in \tT(U)$. Thus $x_2 y \in S(U,\gm).$ Since for every
    $h_2 \in H$ there exists some  $h_1 \in H$ with $h_1 < h_2$, this
    shows that
    $$ \iv(H) \cdot \iv(A) \subset S(U,\gm).$$
    In particular, $S(U,\gm) \neq \emptyset.$  We conclude by
    Remark \ref{rmk:4.13} that the initial transmission  $\al_{U,\gm} : U
    \to
    U_\gm$ in $\STROP$ is not pushout in $\STROP$ (and all the
    more not pushout in $\STROP_m$).
\end{enumerate}
\end{examp}
%

\section{$\h$-transmissions}\label{sec:5}
In \cite[\S6]{IKR2} the equivalence homomorphic relations on a
supertropical semiring  $U$ have been studied in detail. These are
the TE-relations on $U$ such that the supertropical monoid $U/ E$
is a semiring and $\pi_E: U \to U/E$ is a homomorphism of
semirings. It turned out that these relations can be completely
characterized in terms of $U$ as a supertropical monoid, cf.
\cite[Proposition 6.4]{IKR2}, where the crucial  compatibility of
$E$ with addition is characterized in this way.

Having this in mind we define ``\h-transmissions"  for
supertropical monoids,
\begin{defn}\label{defn5.1} We call a map $\al: U \to V$  between
supertropical monoids an \textbf{\h-transmission}, if $\al$ is a
transmission and has also the following property
\begin{align*}& {\HT }  \quad   \forall x,y \in U: \  \text{ If } \ ex < ey \ \text{ and } \ \al(ex) =
\al(ey), \text{ then } \al(y) \in eV.
\end{align*}
\end{defn}
We can read off the following result from \cite[Proposition
6.4]{IKR2}:

\begin{prop}\label{prop5.2} Assume that $U$ and $V$ are
supertropical semirings. Then a map \\ $\al: U \to~V$ is an
\h-transmission iff $\al$ is a semiring homomorphism.
\end{prop}

\begin{rem}\label{rem5.3} We note in passing that in Definition
\ref{defn5.1} the condition $\HT$ can be formally relaxed as
follows.
\begin{align*}& {\HTp } \quad  \forall x,y \in U: \  \text{ If } \ 0<  ex < ey \ \text{ and } \ \al(ex) =
\al(ey), \text{ then } \al(y) \in eV.
\end{align*}
Indeed if $\al$ is a transmission and $ex = 0$,  $\al(ex) =
\al(ey)$, we conclude right away that $0 =  \al(ex) = \al(ey)$,
hence $\al(y) = 0 \in eV$.
\end{rem}

We now study \h-transmissions between supertropical monoids with
the primary goal to gain a more insight into the variety of
homomorphisms between supertropical semirings. If nothing else is
said,  letters   $U, V, W$ will denote supertropical monoids.

\begin{examp}\label{exmp5.4} Every transmission $\al: U \to V$,
such that $\gm:= \al^\nu$ is injective on $(eU) \sm \{0\}$, is an
\h-transmission. Indeed, now the condition $\HTp$ is empty.
\end{examp}

The functorial properties of transmissions stated in Proposition
\ref{prop1.9} have a counterpart for \h-transmissions.

\begin{prop}\label{prop5.5} Let $\al: U \to V $ and $\bt: V \to W$
be maps between supertropical monoids.
\begin{enumerate} \eroman
    \item If $\al$ and $\bt$   are \h-transmissions, then $
    \bt \al $ is an  \h-transmission.

    \item If $\al$ and $\bt \al $  are \h-transmissions and $\al$
    is surjective, then $\bt$  is an \h-transmission.
\end{enumerate}
\end{prop}
\begin{proof}
By Proposition \ref{prop1.9} we may already assume  that $\al$ and
$\bt$ are transmissions.

(i): Assume that $x,y \in U$ are given with $0 < ex < ey$ and $\bt
\al(ex) = \bt \al (ey) $. We have to verify that  $\bt \al (y) \in
eW$.
\begin{description}
    \item[Case 1] $\al (ex) = \al (ey)$. Now $\al(y) \in eV$,
    since  $\al$ is an \h-transmission. This implies  $\bt \al(y) \in eW.$

    \item[Case 2]  $\al (ex) < \al (ey)$. Since $\bt \al (ex) = \bt \al
    (ey)$ and $\bt$ is an \h-transmission, again $\bt \al (y) \in
    eW$.
\end{description}
ii): Let $x,y \in U $ be given with $0 < \al(ex) <
    \al(ey)$ and $\bt \al (ex) = \bt \al
    (ey)$.  Then $0 < ex < ey$. We conclude that $\bt\al(y) \in
    eW$. Since $\al$ is surjective and $\al(ex) = e\al(x)$, $\al(ey) =
    e\al(y)$, this proves that $\bt $ is an \h-transmission.
\end{proof}

\begin{notations}\label{note5.6} $ $

\begin{enumerate} \ealph
    \item We introduce two new categories:
\begin{enumerate}
    \item[i)] Let $\strop_m$ denote the category whose objects are
    the supertropical monoids and morphisms are  the
    \h-transmissions. Notice that this makes sense by Proposition~\ref{prop5.5}.i.

    \item[ii)] Let $\strop$ denote the category whose objects are
    the supertropical semirings  and morphisms are the semiring
    homomorphisms between supertropical semirings.
\end{enumerate}

    \item We further denote  by $\sring$ the category of all
    semirings and semiring homomorphisms.
\end{enumerate}

\end{notations}

$\strop$ is a full subcategory of $\sring$ and, due to Proposition
\ref{prop5.2} also of $\strop_m$. Thus we have the following chart
of categories, where ``$\subset$'' means ``subcategory'' and
``$\subset_{\full}$'' means  ``full subcategory''
$$\begin{array}{ccccc}
  & & \STROP & \subset_{\full} & \STROP_m \\
  &  &\cup & &  \cup  \\
\sring  & \supset_{\full} & \strop & \subset_{\full} & \strop_m .\\
\end{array}$$
Moreover, in slightly symbolic notation,
$$ \strop = \STROP \cap \sring = \STROP_m \cap \sring =  \STROP \cap \strop_m.$$

Our main concern will be to understand relations between
$\strop_m$ and $\STROP$ within the category $\STROP_m$, in order
to get an insight into $\strop = \STROP \cap \strop_m$.
\begin{thm}\label{thm5.7} Assume that $\al : U \to V$  is a
surjective \h-transmission and $U$ is a semiring. Then $V$ is a
semiring.
\end{thm}

\begin{proof} Let $M : = eU$, $N := eV$, and $\gm := \al^\nu$. We check the condition
$\Dis$ in Theorem~\ref{thm1.2} for the supertropical monoid $V.$
Since $\al$ and hence $\gm$ is surjective, this means the
following.

Let $y,z \in U$ and $y' \in M$ be given with $0 < \gm(y') <
\gm(ey)$ and $\gm(y'z) = \gm (eyz)$. Verify that $\al(yz) \in N$!

We have $y' < ey$. If $y'z = eyz$ then $yz \in M$, since $U$ is a
semiring, and we conclude that $\al(yz) \in N$.

There remains the case that $y' z < eyz$. Since $\al$ is an
\h-transmission, we conclude again that $\al(yz) \in N$.
\end{proof}

In the following we assume that $U$ is a supertropical monoid and
$\gm$ is a homomorphism from $M := eU$ onto a (bipotent)  semiring
$N$. We look for \h-transmissions $\al : U \to V$ which cover
$\gm$. We introduce the set
$$ \SigzUgm:= \{ x\in \tT(U) \ds | \exists x_1 \in M : x_1 < ex, \gm(x_1) = \gm (ex) \neq 0 \}. $$
\begin{prop}\label{prop5.8} A transmission $\al : U \to V$
covering $\gm$ is an \h-transmission iff the ghost kernel
$\mfA_\al$ contains the set $\SigzUgm$.
\end{prop}
\begin{proof}
By Scholium \ref{schol4.1} it is evident that $\SigzUgm \subset
\mfA_\al$ iff $\al$ obeys the condition $\HTp$ from above.
\end{proof}

We further introduce the set
$$ \SigUgm:=  \SigzUgm \cup \{ x\in \tT(U) \ds |  \gm (ex) = 0 \},$$
and the supertropical monoids
$$ \brU := U/ E(U, \SigUgm ),$$
$$ U_\gm^h  := \brU/ F( \brU, \gm ),$$
finally the transmission $\alhUgm:U \to U_\gm^h ,$ given by
 \begin{equation}\renewcommand{\theequation}{$*$}\addtocounter{equation}{-1}\label{eq:str.1}
\alhUgm:= \pi_{F(\brU,\gm)} \circ \pi_{ E(U, \SigUgm )}.
\end{equation}
Notice that $\alhUgm$ covers $\gm$ and is the product of an ideal
compression and a strict ghost contraction, so that $(*)$ gives
already the canonical factorization of $\alhUgm$. \{N.B. The ideal
$U \cdot \SigUgm $ contains $\mfa_{U,\gm}$, hence
 $\mfa_{\brU,\gm} \subset M$.\}

 The ghost kernel of $\alhUgm$ contains the set $\SigzUgm$,
 and thus we know by Proposition~\ref{prop5.8}, that
 $\alhUgm$ is an \h-transmission.

 We call $\alhUgm$ a \textbf{pushout in the
 category} $\bf \strop_m$, since the following holds:
\begin{thm}\label{thm5.9} Assume that $\dl: N \to L$ is a
surjective homomorphism from $N$ to a semiring~$L$ and $\bt:U \to
W$ is an \h-transmission covering $\dl \gm$ (in particular $eW
=L$). Then there exists a (unique) \h-transmission $\eta: \Uhgm
\to W$ covering $\dl$ such that $\bt = \eta \circ \alhUgm$.
\end{thm}
\begin{proof}

Let $\al :=  \alhUgm$ and $\lm := \pi_{E(U, \SigUgm)}$. We retain
the notations from above, hence have $ \Uhgm = \brU/ F(\brU, \gm)$
with $\brU = U / E(U, \SigUgm)$.

Now observe that $\SigUgm$ is contained in $\Sig(U, \dl \gm)$.
Indeed, let $x \in \SigUgm$. If $\gm(ex)= 0$, then $\dl \gm(ex)
=0$. If there exists some $x_1 \in M$ with $x_1 < ex$ and
$\gm(x_1) = \gm(ex)$, then either $\dl \gm (ex) = 0$,  or $\dl \gm
(ex) \neq 0$, and then  $x \in \Sig_0(U,\dl \gm)$. Thus  $x \in
\Sig(U, \dl\gm)$ in all cases.

Since $\bt$ is an \h-transmission covering $\dl \gm$, the ghost
kernel of $\bt$ contains $\Sig(U,\dl \gm)$ and hence $\SigUgm$.
Thus we have a factorization
$$\xymatrix{
    \bt:U    \ar[r]_{\lm} & \brU  \ar[r]_{\brBt} & W  \\
}$$ with  $\brBt$ a transmission again covering $\dl \gm$.

We have a commuting diagram (solid arrows)
$$\xymatrix{
    \brU    \ar[drr]_{\pi_{F(\brU,\gm)}}  \ar[drrrr]^{\brBt} & & & &  \\
   U \ar[rr]_\al  \ar[u]^\lm  && U_\gm^h \ar@{-->}[rr]_\eta   & & W\\
   M \ar[rr]_\gm  \ar[u] & &  N \ar[u] \ar[rr]_\dl   & & L \ar[u] &. \\
 }$$
Since $\pi_{F(\brU, \gm)}$ is a pushout in the category $\STROP_m$
(Theorem \ref{thm1.10}), we have a transmission $\eta: \Uhgm \to
W$ covering $\dl$ such that $\eta \circ \pi_{F(\brU, \gm)} =
\brBt$, hence $\eta \circ \al = \brBt \circ \lm = \bt$. Since both
$\al$ and~$\bt$ are \h-transmissions, also $\eta$  is an
\h-transmission (Proposition \ref{prop5.5}.i).
\end{proof}

\begin{cor}\label{cor5.10} (Entering the category $\strop$.)
Let $V : =\Uhgm$ and $\al := \alhUgm : U \to V$. Then $$\htal :=
\sig_V \circ \al : U \ds \to \htV$$ is the \textbf{initial
\h-transmission from $U$ to a semiring covering $\gm$}, i.e.,
given an \h-transmission   $\bt: U \to W$   covering $\gm$ with
$W$ a semiring, there exists a (unique) fiber contraction  $\z:
\htV \to W$ over $N$ with $\bt = \vrp \circ \htal$.
\end{cor}
\begin{proof} By Theorem \ref{thm5.9}, applied with $\dl = \id_N$,
we have fiber contraction $\eta:V \to W $ over $N$ such that $\bt
= \eta \circ \al $. By Theorem \ref{thm4.5} there exists a fiber
contraction $\z: \htV \to W$ over $N$  such that $\eta = \z \circ
\sig_V$. Thus  $\bt = \z \circ \sig_V \circ \al = \z \circ \htal$.
\end{proof}

\begin{thm}\label{thm5.11} Assume that $\al: U \to V$ is a
surjective \h-transmission, and
$$\xymatrix{
    \al:U    \ar[r]_{\lm} & \brU  \ar[r]_{\bt} & W \ar[r]_{\mu} & \brW \ar[r]_{\rho} & V  \\
}$$ is the canonical factorization of $\al$ (cf. \S\ref{sec:2}).
\begin{enumerate} \eroman
    \item The factors $\lm, \bt, \mu, \rho$ are again
    \h-transmissions.

    \item If $U$ is a semiring, then the supertropical monoids $U, W, \brW,
    V$ are semirings, and the maps
     $\lm, \bt, \mu, \rho$ are semiring homomorphisms.
\end{enumerate}
\end{thm}
\begin{proof} i): We know already by Example \ref{exmp5.4} that $\lm, \mu,
\rho$ are \h-transmissions since they cover the identities $\id_M$
and  $\id_L$ respectively  (and $\rho$ is even an isomorphism). We
have $\brU = U/ E(U, \SigUgm)$. Now observe that, if $a'$ and $a$
are elements of $M$ with $a' < a$, $\gm(a') = \gm(a) \neq 0$, then
the fiber $\brU_a = \igm_\brU(a) $ contains no tangible elements.
\{Recall the definition of the set $\SigzUgm \subset \SigUgm$.\}
Thus every transmission $\bt': U \to W' $ covering $\gm$ trivially
obeys the condition $\HTp$ from above (Remark \ref{rem5.3}), hence
is an \h-transmission. In particular, $\bt$ is an \h-transmission.

ii): If $U$ is a semiring, we conclude by Theorem \ref{thm5.7}
successively, that $\brU$, $W, \brW, V$ are semirings. Now invoke
Proposition \ref{prop5.2} to conclude that $\lm, \bt, \mu, \rho$
are semiring homomorphisms.
\end{proof}

We strive for an explicit description of the initial
\h-transmission $\alhUgm$ covering $\gm$. Let $\tH(U,\gm)$ denote
the ideal of $U$ generated by $\SigUgm \cup M$, i.e.,
$$ \tH(U,\gm) := (U \cdot  \SigzUgm) \ds  \cup \mfa_{U,\gm} \ds \cup M.$$
We have
$$ \alhUgm = \pi_{F(\brU, \gm)} \circ \pi_{E(U,\tH)}$$
with $\tH:= \tH(U,\gm)$ and $\brU = U/E(U,\tH)$. Invoking Example
\ref{exmp2.14} we learn that
$$ \alhUgm = \pi_{E(U,\tH(U,\gm),\gm)}.$$
We denote the equivalence relation $E(U,\tH(U,\gm),\gm)$ more
briefly by $H(U,\gm)$.

Our task is to describe this TE-relation explicitly. We will
succeed  if $U$ is a semiring.
\begin{lem}\label{lem5.12}
If $U$ is a semiring, then
$$ \tH(U,\gm) = \SigzUgm \cup \mfa_{U,\gm}\cup M.$$
\end{lem}
\begin{proof}
$\tH(U,\gm)$ contains the set on the right hand side. We are done,
if we verify that a product $xy$ with $x\in \SigzUgm$, $y\in
\tT(U)$, $xy \in \tT(U) \sm \mfa_{U,\gm}$ lies in $\SigzUgm$.

We have $x\in\tT(U)$. By definition of $\SigzUgm$ there exists
some $x' \in M$ with $x' < ex$ and $\gm(x') = \gm (ex) \neq 0$.
Now $x'y \leq exy$, but equality here would imply that $xy \in M$,
since $U$ is a semiring (cf. Theorem~\ref{thm1.2}). Thus $x'y <
exy$. Further $\gm(x'y) = \gm (exy) \neq 0 $. This shows that
indeed $xy \in \SigzUgm$.
\end{proof}

Starting from this lemma and the general description of the
relations $E(U, \mfA, \gm)$ in \\ Theorem~\ref{thm1.14} it is now
easy to write out the TE-relation $H(U,\gm)$. We obtain a theorem
which runs completely  in the category $\strop$.

\begin{thm}\label{thm5.13} Assume that $U$ is a supertropical
semiring. The initial semiring homomorphism $\alhUgm$ covering
$\gm$ is the map
$$\pi_{H(U,\gm)} : U \ds \to U/ H(U,\gm)$$
corresponding to the following equivalence relation $H(U,\gm)$ on
$U$:

 If $x_1,x_2 \in U$, then $x_1 \sim_{H(U,\gm)} x_2$ \ds{iff}
$\gm(ex_1) = \gm(ex_2)$ and either $x_1 = x_2$, or $x_1, x_2 \in M
\cup \SigzUgm$, or $\gm(ex_1)=0$. \endbox
\end{thm}

\section{Ordered supertropical monoids}

In the paper \cite{IKR3} the present authors studied
supervaluations with values in a ``totally ordered supertropical
semiring" \cite[Definition 3.1]{IKR3} and obtained  -- as we
believe -- natural and useful examples of such supervaluations.
This motivates us now to define ``ordered supertropical monoids".

\begin{defn}\label{defn61.1} An \textbf{ordered supertropical
monoid},  or \textbf{$\OST$-monoid} for short, is a supertropical
monoid $U$ equipped with a total ordering $\leq$  of the set $U$,
such that the following hold:

\begin{alignat*}{2}
&(\OST1): \quad&& \text{The ordering $\leq$ is compatible with multiplication, i.e., }\\
& && \text{for $x,y,z \in U$, }  \   x\leq y \ds \Rightarrow xz \leq yz;   \\
 &(\OST2): \quad && \text{The ordering $\leq$ extends the natural total order of the bipotent}\\
& & & \text{semiring $M:= eU$, i.e., if $x,y \in M,$ then }  x \leq y \ds \Leftrightarrow x \leq_M y; \\
&(\OST3): \quad&& 0 \leq 1 \leq e.
 \end{alignat*}

\end{defn}

\begin{lem}\label{lem61.2} Let $x \in U.$
\begin{enumerate} \ealph
    \item Then $ 0 \leq x \leq ex.$
    \item If $x \in \tT(U),$ then $x < ex.$
\end{enumerate}
\end{lem}

\begin{proof} (a): This follows by multiplying the inequality $0 \leq 1 \leq
e$ by $x.$ \pSkip

(b): We have $x \leq ex$ and $x \neq ex$; hence $x < ex.$
\end{proof}

As common,  we call a subset $C$ of a totally ordered set $X$
\textbf{convex}  (in $X$) if for all $x,y\in C$, $z \in X$ with $x
\leq z \leq y$ also $z \in C.$ (This definition still makes sense
if $X$ is only partially ordered, but now we do not need this
generality.)
\begin{prop}\label{prop61.3} $ $

\begin{enumerate} \ealph
    \item For every $c \in M \sm \00$ both the fiber $U_c :=
    \inu_U(c)$ and the tangible fiber $\tT(U)_c := \tT(U) \cap
    U_c$ are convex in $U$.

    \item If $c,d \in M$ and $c <d,$ then $$c < \tT(U)_d < d$$
    (i.e., $c < x < d$ for every $x\in \tT(U)_d$).
\end{enumerate}

\end{prop}

\begin{proof} (a): Let $x,y \in U_c$, $z\in U,$ and $x \leq z \leq
y.$ We can conclude from $c  =ex \leq ez \leq ey = c$ that $ez
=c;$ hence $z \in U_c.$ Now assume that in addition $x,y \in
\tT(U).$ If $z$ were ghost, hence $ez =c$, it would follow by
Lemma \ref{lem61.2}.b that $y < ey  = z.$ Thus $z \in \tT(U).$
\pSkip

(b): Let $x \in \tT(U)_d.$ Then $x < ex = d$,  by Lemma
\ref{lem61.2}.b. Suppose $x \leq c$. Then it would follow that $ex
= d \leq c,$ which is not true. Thus $c < x.$
\end{proof}

\begin{thm}\label{thm61.4}
If $(U, \leq)$ is an $\OST$-monoid, then $U$  is a semiring.
\end{thm}

\begin{proof} We verify condition $\Dis$ in Theorem \ref{thm1.2}.
Let $x,y \in U,$ $x' \in M,$ and assume that $x' < ex,$ but $x'y =
exy.$ From $x' < ex$  we conclude by Proposition \ref{prop61.3}.b
that $x' \leq x$. Furthermore $x \leq ex$ by Lemma
\ref{lem61.2}.a. Multiplying by $y,$ we obtain
$$x'y \leq xy \leq exy \leq x'y, $$
and we conclude that $xy = exy.$
\end{proof}

\begin{thm}\label{thm61.5}
If $(U, \leq)$ is an $\OST$-monoid, then addition\footnote{Recall
the  formulas for $x + y$ in \S\ref{sec:1}, preceding Theorem
\ref{thm1.2}.} in the semiring  $U$ is compatible with the
ordering $\leq,$ i.e., $(x,y,z\in U)$
$$ x \leq y \dss \Rightarrow x+z \leq y + z.$$
\end{thm}

\begin{proof} We conclude from  $x \leq y$ that $ex \leq ey.$ We
distinguish the cases $ex < ey$ and $ex = ey$, and go through
various subcases.
\begin{description}
    \item[Case 1] $ex < ey.$
\begin{enumerate} \ealph
    \item If $ez \leq ex,$ then $e(x + z) = ex + ez = ex$ and $y + z =
    y.$ Since $ex < ey,$ we conclude that $x + z \leq e(x+z) < y +
    z$ (cf. Proposition \ref{prop61.3}.b).

    \item If $ex <  ez < ey,$ then $x + z = z$, $y + z = y$, and we
    conclude from $ ez < ey$
     that $x + z  < y +
    z$.

    \item If $ex  = ey,$ then $x + z = z$, $y + z = ey$. Since  $ z
    \leq
    ez$ we obtain
     that $x + z  \leq y +
    z$.

    \item If $ey  < ez,$ then $x + z = z$, $y + z = z$, hence     $x + z  =  y +
    z$.

\end{enumerate}
\pSkip

    \item[Case 2] $ex = ey.$

\begin{enumerate} \ealph
    \item If $z < ex,$ then $x + z = x,$ $y + z = y.$
    \item If $z = ex,$ then $x + z = ex,$ $y + z = ey.$
    \item If $ ex < z,$ then $x + z = y + z = z.$
\end{enumerate}
Thus in all three cases $x + z \leq y + z. $

\end{description}
\end{proof}

Starting from now we denote an $\OST$-monoid $(U, \leq)$ by the
single letter $U$. From Theorems \ref{thm61.4} and \ref{thm61.5}
it  is obvious that  the present $\OST$-monoids are the same
objects as the totally ordered supertropical semirings defined in
\cite[Definition 3.1]{IKR3}. Examples of these structures can be
found in \cite[\S3, \S4, \S6]{IKR3}.

\begin{defn}[{= \cite[Definition 5.1]{IKR3}}]\label{defn61.6}
Assume that $U$ and $V$ are $\OST$-monoids. We call a transmission
$\al: U \to V$ (cf. Definition \ref{defn1.3}) \textbf{monotone},
if $\al $ is compatible with the ordering on $U$ and $V$, i.e.,
$$ \forall x,y \in U: \quad x \leq y \dss \Rightarrow \al(x) \leq \al(y).$$
\end{defn}

\begin{thm}[{cf. \cite[Theorem 5.3]{IKR3}}]\label{thm61.7}
Every monotone transmission $\al:   U \to V$ is a semiring
homomorphism.
\end{thm}

\begin{proof}
We verify condition $\HT$ in Definition  \ref{defn5.1}, and then
will be done by Proposition \ref{prop5.2}.

Let $x,y \in U$ with $ex < ey$ and $\al(ex) = \al(ey)$. By
Proposition \ref{prop61.3} we have $ex < y < ey.$ Applying $\al$,
we obtain
$$ \al(ex) \ds  \leq \al(y)  \ds \leq \al(ey) \ds  = \al(ex),$$
hence $\al(ex) = \al(ey)$. But $\al(ey) = e\al(y)$ (cf. Definition
\ref{defn3.1}), and we conclude that $\al(y) \in eV,$ as desired.
\end{proof}

Another proof, which relies more on the semiring structure of $U $
and $V$, can be found in \cite[\S5]{IKR3}.

\begin{examp}\label{defn61.8}
Every bipotent semiring can be regarded as an $\OST$-monoid. (This
is the case  $1 = e$.) If $U$ is an $\OST$-monoid, $M = eU$ (our
present overall assumption), then $\nu_U: U \to M $ is a monotone
transmission.
\end{examp}

We indicate a  way how to obtain new $\OST$-monoids from given
ones. First we quote a general fact about total orderings (cf.
e.g. \cite[Remark 4.1]{IKR2}).

\begin{lemdef}\label{deflem61.9} Let $X$ be a totally ordered set
and $f: X \twoheadrightarrow Y$ a map from $X$ onto a set  $Y.$
Then there exists a (unique)  total ordering  on $Y$, such that
$f$ is order preserving, iff all fibers $f^{-1} (y),$ $y\in Y,$
are convex in $X.$ We call this total ordering the
\textbf{ordering on $Y$ induced by $f$}.
\end{lemdef}

N.B. This ordering on $Y$ can be characterized as follows: For
$x_1, x_2 \in X$
$$ f(x_1) <  f(x_2) \dss \Rightarrow x_1 < x_2 \dss \Rightarrow f(x_1) \leq f(x_2). $$
Alternatively, we can state:
$$
\begin{array}{ll}
  \text{If } x_1 < x_2,  & \text{then } f(x_1) \leq f(x_2), \\
  \text{If } x_1 > x_2,  & \text{then } f(x_1) \geq f(x_2). \\
\end{array}
$$

\begin{thm}\label{thm61.10}
Assume that $U$ is an $\OST$-monoid, $V$  is a supertropical
monoid, and $\al: U \to V$ is a surjective transmission. Assume
further that for every $p \in V$ the fiber $\ial(p)$ is convex in
$U$. Then $V$, equipped  with the total ordering induced by $\al$,
is again an $\OST$-monoid.
\end{thm}

\begin{proof} We verify the axioms $\OST1$-$\OST3$ in Definition
\ref{defn61.1} for the induced ordering $\leq_V$ on~$V.$

$(\OST 1):$ Let $x,y,z \in U$ and $\al(x) < \al(y)$. Then $x < y,$
hence $xz \leq yz$, hence
$$ \al(x) \al(z) \ds = \al(xz) \ds {\leq_V} \al(yz) \ds = \al(y)  \al(z).$$
\pSkip

$(\OST 2):$ Let $M := eU,$ $N := eV. $ On $U$ and $V$ we have the
given orderings $\leq_U,$ $\leq_V$, and on $M$ and $N$ we have the
natural orderings $\leq_M,$ $\leq_N$ as bipotent semirings. The
ordering~$\leq_U$ restricts on $M$ to $\leq_M.$ We have to verify
that $\leq_V$ restricts on $N$ to  $\leq_N.$

The map $\al : U \to V$ restricts to a semiring homomorphism $\gm:
M \to N$, which consequently is compatible with
 $\leq_M$ and  $\leq_N$. Let $x,y \in M.$ If $\al(x) <_V \al(y)$
 then $x < _U y,$ hence $x <_ M y$, hence $\gm(x) \leq _N \gm(y).$
 Thus
 $$ \al(x) \leq _V \al(y) \dss \Rightarrow \gm(x) \leq _N \gm(y).$$
Conversely, if $\gm(x) < _N \gm(y)$, then $x < _M y$, hence $x <_U
y$, hence $\al(x) \leq _V \al(y).$ Thus
 $$ \gm(x) \leq _N \gm(y) \dss \Rightarrow \al(x) \leq _V \al(y).$$
This proves that indeed the ordering  $\leq_V$ restricts to
$\leq_N$ on $N.$

\pSkip

 $(\OST 3):$ Applying $\al$ to $0 \leq 1 \leq e$ in $U$, we
obtain $0 \leq 1 \leq e$ in $V.$
\end{proof}

We are ready for the main result of this section, which roughly
states that, given a monotone transmission $\al : U \to V,$ the
canonical factors of $\al$ may be viewed as monotone transmissions
in a unique way. We will relay on three easy lemmas.

\begin{lem}\label{lem61.11}
Assume that $U , V, W$ are  $\OST$-monoids,  and $\al: U \to V,$
$\bt: V \to W$ are  transmissions. Assume further that $\al$  and
$\bt \circ \al$ monotone and $\al$ surjective. Then $\bt$ is
monotone.
\end{lem}

 \begin{proof} If $x,y \in U$ and $\al(x) < \al(y)$, then $x < y$,
 hence $\bt\al(x) \leq \bt\al(y).$ Thus   $\al(x) \leq
 \al(y)$ implies  $\bt(\al(x)) \leq \bt(\al(y))$.
\end{proof}

\begin{lem}\label{lem61.12}
Assume that  $\al: U \to V$ is a monotone transmission (between
$\OST$-monoids). Let $\mfA$ denote the ghost kernel of $\al,$
$\mfA = \mfA_\al$. Then for any $c \in eU$ the fiber $\mfA_c :=
\mfA \cap U_c$ is an upper set of the totally ordered set $U_c.$
\end{lem}

\begin{proof}
Assume that $x \in \mfA_c,$ $y \in U_c,$ and $x < y.$ Then $x < y
\leq c,$ hence  $\al(x) \leq  \al(y) \leq \al(c).$ Since $x$ lies
in the ghost kernel $\mfA$ of $\al$ we have  $$\al(x) \ds = e
\al(x) \ds =  \al(ex) \ds = \al(c).$$ It follows that $\al(y) =
\al(c) \in eV,$ hence $y \in \mfA,$ hence  $y \in \mfA_c.$
\end{proof}

\begin{lem}\label{lem61.13}

Assume that  $\al: U \to V$ is a monotone transmission with
trivial ghost kernel. Let $\gm = \al^\nu:M\to N$  denote the ghost
part of $\al.$ Then $U_c = \{ c\}$ for any $c \in M$ such that
there exists some $c_1 < c$ in $M$ with $\gm(c_1) = \gm(c).$
\end{lem}

\begin{proof}
Precisely this has  been verified in the proof of Theorem
\ref{thm61.7}.
\end{proof}

\begin{thm}\label{thm61.14}
Assume that $U , V$ are an $\OST$-monoids,  and $\al: U \to V$ is
a  surjective monotone transmission. Assume further that
$$\xymatrix{
    \al:U    \ar[r]_{\lm} & \brU  \ar[r]_{\bt} & W \ar[r]_{\mu} & \brW \ar[r]_{\rho} & V  \\
}$$ is a canonical factorization (cf. \S\ref{sec:2}) of the
transmission $\al$. Then the monoids $\brU, W, \brW$ can be
equipped with total orderings (in a unique way), such that they
become $\OST$-monoids and all factors $\lm, \bt, \mu, \rho$ are
monotone  transmissions.
\end{thm}

\begin{proof}
a) Let $\gm := \al^\nu: M \to N$ denote the ghost part of the
transmission $\al$ and $\mfA$ denote the ghost kernel of $\al.$
Without loss of generality we may assume that
$$  \brU = U/ E(U,\mfA), \qquad \lm = \pi_{E(U,\mfA)}, \qquad W=
\brU/{F(\brU,\gm)}, $$ $$   \bt = \pi_{F(\brU,\gm)}, \qquad \brW =
V, \qquad \rho = \id_V.$$

For any $c\in M$ we have $\ilm(c) = \mfA_c,$ which by Lemma
\ref{lem61.12} is an upper set $U_c,$ hence is convex in $U_c.$
Since $U_c$ is convex in $U$, it follows that $\ilm(c)$ is convex
in $U.$

Invoking Lemma \ref{deflem61.9}, we equip the monoid $\brU$ with
the total ordering induced by $\lm,$ and then know by Theorem
\ref{thm61.10} that $\brU$ has become an $\OST$-monoid and $\lm$
has become a monotone transmission. By Lemma \ref{lem61.11} also $
\mu \circ \bt: \brU \to V$ is monotone. \pSkip

b) Replacing $U$ by $\brU,$ we are allowed to assume henceforth
that $\al: U \to V$ has trivial ghost kernel, and may focus on the
canonical factorization $\al = \mu \circ \bt$ with $ \bt =
\pi_{F(U,\gm)}$ and $\mu : W \to V$ a tangible fiber contraction.

We use the identifications in Convention \ref{convs3.2}.b to
handle $W = U/{F(U,\gm)}$ and $\bt = \pi_{F(U,\gm)}$. For any
$d\in N = eW$ the tangible fiber $\tT(W)_d$ is the union of all
fibers $\tT(U)_c$ with $c \in  \igm (d).$ Let $L(\gm)$ denote
those $c \in M$ such that $c \neq 0$ and $c$ is the smallest
element of $\igm(\gm(c))$. Lemma \ref{lem61.13} tells us that
$\tT(U)_c \neq \emptyset$ if $ c \in M\sm L(\gm).$ Thus we have
the following picture: If $d\in N$ then $\tT(W)_d = \tT(U)_c$ if
there exists $c\in L(\gm)$ with $\gm(c) =d,$ and this $c$ is then
unique. Otherwise $\tT(W)_d = \emptyset.$ \pSkip

c) Looking again at Convention \ref{convs3.2}.b we see that $\bt$
has the following fibers:  If $p \in \tT(W)_d$, $d = \gm(c)$ with
$c\in L(\gm)$, then $\ibt(p) = \{ p\}$ (using the identifications
in Convention \ref{convs3.2}.b).  If~$d\in N$, then $\ial(d)=
\igm(d).$ Recalling Proposition \ref{prop61.3}, we see that
$\igm(d)$ is convex in~$U.$ Thus all fibers of $\bt$ are convex in
$U$.

Invoking again Lemma \ref{deflem61.9} and Theorem \ref{thm61.10},
we equip $W$ with that total ordering induced by $\bt$, which
makes $W$ an $\OST$-monoid and $\bt$ a monotone transmission. By
Lemma \ref{lem61.11} we conclude that also $\mu$ is a monotone
transmission.
\end{proof}

\section{$\m$-supervaluations }\label{sec:6}
\begin{defn}\label{defn6.1}

Let $R$ be a semiring. An \textbf{\m-supervaluation on $R$} is a
map $\vrp: R \to U$  to a supertropical monoid which fulfills the
axioms $\SV1$-$\SV4$ required for a  supervaluation in
\cite[Definition 4.1]{IKR1}, there for $U$ a supertropical
semiring. To repeat,
 \begin{alignat*}{2}
&\SV1:\ &&\varphi(0)=0,\\
&\SV2:\ &&\varphi(1)=1,\\
&\SV3:\ &&\forall a,b\in R: \varphi(ab)=\varphi(a)\varphi(b),\\
&\SV4:\ &&\forall a,b\in R: e\varphi(a+b)\le
e(\varphi(a)+\varphi(b))\quad
[=\max(e\varphi(a),e\varphi(b))].\end{alignat*}
We then say that $\vrp$ \textbf{covers} the \m-valuation
$$ e \vrp : R \ds \to eU, \qquad a \mapsto e\vrp(a).$$
\end{defn}

Most notions developed for supervaluations in \cite{IKR1},
\cite[\S2]{IKR2} make sense for \m-supervaluations in the obvious
way and will be used here without further explanation, but we
repeat the definition of dominance.

\begin{defn}\label{defn6.2}
Assume that $\vrp: R \to U$ and $\psi: R \to V$ are
\m-supervaluations. We say that $\vrp$ \textbf{dominates} $\psi$
and write $\vrp \geq \psi$, if for all $a, b \in R$ the following
holds.
\begin{alignat*}{3}
&\D1.\quad && \varphi(a)=\varphi(b)& &\Rightarrow \ \psi(a)=\psi(b),\\
&\D2.\quad && e\varphi(a)\le e\varphi(b)&&\Rightarrow \  e\psi(a)\le e\psi(b),\\
&\D3.\quad && \quad\varphi(a)\in eU && \Rightarrow \  \psi(a)\in
eV.
\end{alignat*}
\end{defn}

If $\vrp: R \to U$ is an \m-supervaluation and $\al: U \to V$ is a
transmission, then clearly $\al \circ \vrp$ is an
\m-supervaluation dominated by $\vrp$.

Conversely, if $\vrp:R \to U$  is an \m-supervaluation which is
surjective (i.e., $U = \vrp(R) \cup e \vrp(R)$, cf.
\cite[Definition 4.3]{IKR1}), and $\psi:R \to V$ is an
\m-supervaluation dominated by~$\vrp$, there exists a (unique)
transmission $\al:  U \to V$ with $\psi = \al \circ \vrp$. This
can be proved in exactly the same way as done in \cite[\S5]{IKR1}
for supervaluations. If $\vrp$ and $\psi$ cover the  same
\m-valuation $v: R \to M$, then $\al$ is a fiber contraction over
$M$ (hence an \h-transmission). \pSkip

Let now $v: R \to M$ be a fixed \m-valuation. We call any
\m-supervaluation $\vrp$ with $e \vrp = v$ an \textbf{\m-cover} of
$v$. We call two \m-covers $\vrp: R \to U$ and $\psi: R \to V$
\textbf{equivalent}, if $\vrp \geq \psi$  and $\psi \geq \vrp$. If
$\vrp$ and $\psi$  are surjective this means that $\psi = \al
\circ \vrp$ with $\al : U \to V$ an isomorphism over $M$.

We further denote the equivalence class of an \m-cover $\vrp$ of
$v$ by $[\vrp]$, and the set of all these classes by
$\Cov_m(\vrp)$. This set is partially ordered by declaring
$$ [\vrp] \geq [\psi] \dss{\text{iff}}  \vrp \geq \psi.$$
\pSkip

We now assume for simplicity and without loss of generality
\emph{that $v$ is surjective}. Then every class $\z \in \Cov_m(v)$
can be represented by a surjective \m-supervaluation.

\begin{prop}\label{prop6.3} If $\vrp: R \to U$ is an \m-cover of
$v$, the subset
$$C(\vrp) := \{ [\psi] \in \Cov_m(v) \ds | \vrp \geq \psi \} $$
of the poset $\Cov_m(v)$ is a complete lattice. It has the top
element $[\vrp]$ and the bottom element~$[v]$.
\end{prop}

\begin{proof} We may assume that the \m-supervaluation $\vrp: R \to
U$ is surjective. Let $\MFC(U)$ denote the set of all
$\MFC$-relations on $U$. This set is partially ordered by
inclusion,
$$ E_1 \leq E_2  \dss{\text{iff}} E_1 \subset E_2.$$
\{We view the equivalence relations $E_i$ as subsets of $U \times
U$ in the usual way.\} We have a bijection
$$ \MFC(U) \ds{\tilde{\longrightarrow}} C(\vrp), \qquad E \mapsto [\pi_E \circ \vrp], $$
since every fiber contraction $\al$ over $M$ is of the form $\rho
\circ \pi_E$, with $E \in \MFC(U)$ uniquely determined by $\al$
and $\rho$ an isomorphism over $M$. Clearly the bijection reverses
the partial orders on $\MFC(U)$ and $C(\vrp)$. Now it can be
proved exactly as in \cite[\S7]{IKR1} for $U$ a supertropical
semiring, that the poset $\MFC(U)$ is a complete lattice. Thus
$C(\vrp)$ is a complete lattice.
\end{proof}

We  construct  a supertropical monoid $U$ which will be the target
of an \m-cover \\ $\vrp: R \to U$  dominating all other \m-covers.

Let $\mfq := \iv(0) = \supp(v)$. As a set we define  $U$ to be the
disjoint union of $R \sm \mfq$ and  $M$,
$$ U = (R \sm \mfq) \ds{\dot \cup}  M.$$
We introduce on $U$ the following multiplication: For $x,y \in U$
$$
 x \Udot y = y \Udot x = \left\{
\begin{array}{lll}
x \Rdot y  &  \text{if} & x,y, xy \in R \sm \mfq,  \\[1mm]
0  &  \text{if} & x,y \in R \sm \mfq, \ xy \in \mfq,   \\[1mm]
v(x) \Mdot y    &  \text{if} & x \in R \sm \mfq, \ y \in M,   \\[1mm]
x \Mdot y  &  \text{if} & x,y \in M.   \\
\end{array}
 \right.
$$
It is readily checked that $U$ with  this multiplication is a
monoid with unit element $1_U = 1_R$ and absorbing  idempotent
$0_U = 0 _M$. Moreover $e:= 1_M$ is an idempotent of $U$ such that
$M = e \cdot U $ and $M$ in its given multiplication is a
submonoid of $U$. Finally $0_M$ is the only element $x$ of $U$
with $0_M \cdot x = 0_M$. Thus, if we choose the given total
ordering on the submonoid $M$ of $U$, we have established on $U$
the structure of a supertropical monoid  (cf.~
Definition~\ref{defn1.1}). We denote this supertropical monoid now
by $\Uz(v)$.

\begin{thm}\label{thm6.4} $ $
\begin{enumerate} \eroman
    \item The map $\vrpzv:R \to \Uz(v)$ with
$$\vrpzv(a) = \left\{
\begin{array}{ll}
  a & \text{if } \ a \in R \sm \mfq, \\
  0_M & \text{if }\  a \in \mfq, \\
\end{array}
\right.
$$
     is a
    surjective \m-valuation covering $v$.

    \item $\vrpzv$ dominates every other \m-cover of $v$.
\end{enumerate}
\end{thm}

\begin{proof} (i): An easy verification.

(ii): Let $\psi: R \to V $ be an \m-cover of $v$. We define a map
$\al: \Uz(v) \to V$ by $\al(a) = \psi(a)$ for  $a \in R\sm \mfq$
and $\al(x) = x$ for $x \in M$. It then can be verified in a
straightforward way  that $\al$ is a transmission and $\al \circ
\vrpzv = \psi$.
\end{proof}

\begin{cor}\label{cor6.5}
The poset $\Cov_m(v)$ is a complete lattice with top element
$[\vrpzv]$.
\end{cor}

\begin{proof} By Theorem \ref{thm6.4} we have $\Cov_m(v) =
C(\vrpzv)$, and this is a complete lattice by Proposition
\ref{prop6.3}.
\end{proof}

$v: R \to M $ itself may be regarded as an \m-supervaluation (in
fact a supervaluation) covering $v$, and thus $[v]$ is the bottom
element of $\Cov_m(v).$ \pSkip

In \cite{IKR1} we had introduced the poset $\Cov(v)$ consisting of
the equivalence classes $[\vrp]$ of supervaluations $\vrp : R \to
U$ with $e \vrp = v,$ and we called these supervaluations $\vrp$
the \textbf{covers} of the \m-valuation $v$. Thus, a surjective
\m-cover $\vrp: R \to U $ of $v$ is a cover of $v$ iff $U$ is a
semiring. The set $\Cov(v)$  is a subposet of $\Cov_m(v)$.

\begin{prop}\label{prop6.6} If $\z \in \Cov(v)$, $\eta \in
\Cov_m(v)$ and $\z \geq \eta$, then $\eta \in \Cov(v)$.
\end{prop}

\begin{proof}
We choose surjective \m-valuations $\vrp: R \to U$, $\psi: R \to
V$ with $U$ a semiring and $\z = [\vrp]$, $\eta = [\psi]$. There
exists a fiber contraction $\al: U \onto V$ over $M$ with $\al
\circ \vrp = \psi$.  Since~$\vrp$ and $\psi$ are surjective, $U =
\vrp(R) \cup e \vrp(R)$ and $V = \psi(R) \cup e \psi (R)$. We
conclude that
$$ \al (U) = \al \vrp(R) \cup e \al \vrp(R) = \psi(R) \cup e \psi (R) = V.$$
Since $U$ is a semiring, it follows by Theorem \ref{thm5.7}, or
already by Theorem \ref{thm1.5}.ii, that $V$ is a semiring, hence
$\eta \in \Cov(v).$
\end{proof}

Recall from \S\ref{sec:4} that every supertropical monoid $U$
gives us a supertropical semiring \\ $\htU = U / E(U,S(U))$
together with an ideal compression $\sig_U := \pi_{E(U,S(U))}: U
\onto \htU$. Here $S(U)$ is the set of tangible NC-elements in $U$
(cf. Definition \ref{defn4.2}).

\begin{defn}\label{defn6.7} For every \m-supervaluation $\vrp: R \to
U$ we define a supervaluation $$\htvrp := \sig_U \circ \vrp : R
\to \htU.$$
\end{defn}

\begin{prop}\label{prop6.8}
Let $\vrp$ be an \m-cover of $v$.
\begin{enumerate} \eroman
    \item $\htvrp$ is a cover of $v$ and $\vrp \geq \htvrp$.
    \item If $\psi$ is a cover of $v$ with $\vrp \geq \psi$, then
    $\htvrp \geq \psi.$
    \item If $\psi$ is an \m-cover of $v$ with $\vrp \geq \psi$,
    then $\htvrp \geq \htpsi$.
\end{enumerate}
\end{prop}

\begin{proof} i): This is obvious.

ii): We may assume that $\vrp$ is a surjective \m-supervaluation.
Then we have a fiber contraction $\al : U \to V$ over $M$ with
$\psi = \al \circ \vrp$. By Theorem \ref{thm4.5} we have a
factorization $\al = \bt \circ \sig_U$ with $\bt$ another fiber
contraction over $M$. We conclude that $\psi = \bt \sig_U \vrp =
\bt \htvrp$.

iii): $\vrp \geq \psi \geq \htpsi$ by i), hence $\htvrp \geq
\htpsi$ by ii).
\end{proof}

\begin{thm}\label{thm6.9}
As before assume that $v: R \to M$ is a surjective \m-valuation.
Let $$U(v) := (\Uz(v))^\wedge$$ and
$$ \vrp_v := (\vrpzv)^\wedge : R \to U(v).$$
The supervaluation $\vrp_v$ is an \textbf{initial cover of} $v$;
i.e., given any supervaluation $\psi: R \to V$ covering $v$, there
exists a (unique)
 semiring homomorphism $\al : U \to V$ over $M$ with $\psi = \al \circ \vrp_v.$\end{thm}

\begin{proof}
We may assume that $\psi$ is a surjective supervaluation covering
$v$. We know by Theorem~\ref{thm6.4} that $\vrpzv \geq \psi$, and
conclude by Proposition~\ref{prop6.8}.ii that  $\vrp_v =
(\vrpzv)^\wedge \geq \psi.$ Thus, there exists a fiber contraction
$\al: U(v) \to V$ over $M$ with $\psi = \al \circ \vrp$. Since
$U(v)$ and $V$ are semirings, $\al$  is a semiring homomorphism
over $M$ (cf. Proposition \ref{prop5.2}).
\end{proof}

\begin{cor}\label{cor6.10} The poset $\Cov(v)$  is a complete
lattice with top element $[\vrp_v]$.
\end{cor}

We proved in \cite[\S7]{IKR1} that $\Cov(v)$ is a complete
lattice, but -- except  in the case that~$v$ is a valuation --
there we have only proved that the a top element $[\vrp_v]$ exists
(loc. cit, Proposition~7.5), without giving an explicit
description of $\vrp_v$. Starting from the formula $\vrp_v =
(\vrpzv)^\wedge$, this is now possible.

Let $U := \Uz(v)$. Then $\tT(U) = R \sm \mfq$  and $eU= M$. We
have
$$ \htU = U/ E (U, \tD(U)) = U / E(U, S(U))$$
with $S(U)$ the set of tangible NC-products in $U$ (cf.
Definition~\ref{defn4.2}) and $\tD (U) = S(U) \cup M$, which is an
ideal of $U$ (cf. \S\ref{sec:4}). We view $\htU$ as a subset of
$U$, as indicated in Convention~\ref{convs3.2}.a. Thus $e \htU =
M$ and $\tT (\htU) = \tT(U) \sm S(U)$. Recalling the description
of the supertropical monoid $\Uz(v)$ from above, we see that
$S(U)$ is the following subset $Y(v)$ of $R \sm \mfq$:
$$ Y(v) := \{ ab \ds |  a,b \in R, \exists a' \in R \text{ with } v(a)  < v(b) , v(a'b) = v(ab) \neq 0 \}.$$
Thus $\tT(\htU) = R \sm \mfq'$ with $\mfq' := \mfq \cup Y(v)$.
Clearly $R \cdot Y(v) \subset \mfq \cup Y(v)$, and thus $\mfq'$ is
an ideal of $R$.

Looking again at Convention \ref{convs3.2}.a we obtain a
completely explicit description of $U(v)$ and~$\vrp_v$ as follows.
\begin{schol}\label{schl6.11}
 Assume that $v: R \to M$ is a surjective \m-valuation
with  support $\mfq = \iv(0)$. Let $\mfq' := \mfq \cup Y(v)$ with
the set $Y(v) \subset R \sm \mfq$ given above. Then $ U(v)$ is the
subset $(R \sm \mfq') \cup M$ of  $\Uz(v) = (R \sm \mfq) \cup M$,
with the following new product $\odot$: If $x,y \in U(v)$, then
$$ x \odot y = \left\{
\begin{array}{ll}
  xy & \text{if }  x,y, xy \in R \sm \mfq',\\
exy & \text{otherwise},\\
\end{array}
\right. $$ the products on the right hand side taken in $\Uz(v)$.
\{Here $e = e_{\Uz(v)} = e_{U(v)}$.\} The initial covering
$\vrp_v: R \to U(v)$ is given by
$$ \vrp_v(a) = \left\{
\begin{array}{ll}
  a & \text{if }  a \in R \sm \mfq',\\
v(a) & \text{if }  a\in \mfq'.\\
\end{array}
\right. $$
\end{schol}\label{schol6.11}
\begin{thm}\label{thm6.12}
As before assume that $v: R \to M$ is a surjective \m-valuation.
\begin{enumerate} \eroman
    \item If $\vrp$ is any supervaluation covering $v$, then
    $\vrp(a)$ is ghost for every $a \in Y(v)$.
    \item If $Y(v) = \emptyset$, then every surjective
    \m-supervaluation covering $v$ is a supervaluation. In short
    $\Cov_m (v) = \Cov(v)$.

\end{enumerate}
\end{thm}

\begin{proof} i): We read off from Scholium \ref{schl6.11} that
$\vrp_v(a) = v(a) \in M$ for any $a \in Y(v)$. Since~$\vrp_v$
dominates $\vrp$, also $\vrp(a)= v(a) \in M$. \pSkip

ii): If $Y(v)= \emptyset$, it follows from Scholium \ref{schl6.11}
that $\vrpzv = \vrp_v$. Thus $[\vrpzv] \in \Cov(v)$. If $\vrp$ is
any \m-supervaluation covering $v$, then $\vrpzv \geq \vrp$, hence
$[\vrp] \in \Cov(v)$ by Proposition \ref{prop6.6}.
\end{proof}

In particular $Y(v)$ is empty if $v$ is a valuation, since this
means that the bipotent semiring~$M$ is cancellative. Then we fall
back on the explicit description of the initial cover  $\vrp_v$ in
\cite{IKR1} which in present notation says that $\vrp_v  = \vrpzv$
(loc. cit, Example 4.5). \pSkip

In large parts of \cite{IKR1} and whole \cite[\S2]{IKR2}, where we
studied  coverings of valuations, it was important that we have
tangible supervaluations at our disposal. There remains the
difficult task to develop a similar theory for coverings of
\m-valuations, which are not valuations. We leave this to the
future. But we mention that there exist many natural and beautiful
\m-valuations which are not valuations, as is already clear from
\cite{HV} and \cite{Z}. More on this can be found in a recent
paper \cite{IKR3}.

\section{Lifting ghosts to tangibles}\label{sec:7}

\begin{defn}\label{defn7.1}
We call a supertropical monoid $U$ \textbf{unfolded}, if the set
$\tT(U)_0 := \tT(U) \cup \{0 \}$ is closed under multiplication.
\end{defn}

If $U$ is unfolded, then $N:= \tT(U)_0$ is a monoid under
multiplication with absorbing  element~$0$. Further $M := eU$ is a
totally ordered monoid with absorbing  element $0$, and the
restriction
$$ \rho := \nu_U |U : N \ds \to M$$
is a monoid homomorphism with $\irho (0) = \{ 0 \}$. Observing
also that $e_U = 1_M = \rho(1_N)$, we see that the supertropical
monoid $U$ is  completely determined by the triple $(N,M, \rho)$.
This leads to a way to construct all unfolded supertropical
monoids up to isomorphism.
\begin{construction} \label{constr7.2} Assume that we are given a
totally ordered monoid $M$ with absorbing element $0_M \leq x$ for
all $x \in M$, i.e., a bipotent semiring $M$, further an (always
commutative) monoid $N$ with  absorbing element  $0_N$, and a
multiplicative map $ \rho : N \to M$ with $\rho (1_N) =
\rho(1_M)$, $\irho(0_M) = \{ 0 _N\} $. Then we define an unfolded
supertropical monoid $U$ as follows: As a set $U$ is the disjoint
union of $M \sm \{ 0_M\}$, $N \sm \{ 0_N\}$, and a new element
$0$. We identify $0_M = 0_N = 0$, and then write
$$U = M \cup N, \quad \text{with } M \cap N = \{ 0 \}.$$

The multiplication on $U$ is given by the rules, in obvious
notation,
$$  x \cdot y = \left\{
\begin{array}{lll}
  x \Ndot y & \text{if} &  x,y \in N, \\[1mm]
  \rho(x) \Mdot y & \text{if} &  x \in N, y \in M, \\[1mm]
  x \Mdot \rho(y) & \text{if} &  x \in M, y \in N, \\[1mm]
  x \Mdot y & \text{if} &  x,y \in M.  \\
\end{array}
\right.
$$
It is easy to verify that $(U, \cdot \; )$ is a (commutative)
monoid with $1_U = 1_N$ and absorbing element $0$. Let $e := 1_M$.
Then $eU = M$ and $\rho(x) = ex$ for $x \in M $,  further   $ex =
0$ iff $x=0$ for any $x \in U$, since $\irho(0) = \{ 0\}$. Thus
$(U, \cdot \; , e)$ ,together with the given ordering on $M = eU$,
is a supertropical monoid. It clearly is unfolded. We denote this
supertropical monoid $U$ by $\STR(N,M, \rho)$.
\end{construction}

This construction generalizes the construction of supertropical
domains \cite{IKR1} (loc. cit. Construction 3.16). There we
assumed that $N \sm \{ 0 \}$ and $M \sm \{ 0 \}$ are closed under
multiplication, and that the monoid $M \sm\{ 0 \}$ is
cancellative, and we obtained all supertropical predomains up to
isomorphism. Dropping here just the cancellation hypothesis would
give us a class of supertropical monoids not broad enough for our
work below.

The present notation $\STR(N,M, \rho)$ differs slightly from the
notation $\STR(\tT, \tG, v)$ in \cite[Construction 3.16]{IKR1}.
Regarding the ambient context this should not cause confusion.

\pSkip

We add a description of the transmissions  between two unfolded
supertropical monoids.

\begin{prop}\label{prop7.3} Assume that $U' = \STR(N', M', \rho')$
and $U = \STR(N, M, \rho)$ are unfolded supertropical monoids.
\begin{enumerate} \eroman
    \item If $\lm: N' \to N$ is a monoid homomorphism with $\lm(0) =
    0$, and $\mu: M' \to M$ is a semiring homomorphism, and if $\rho' \lm = \mu
    \rho$, then the well-defined map
    $$ \STR(\lm,\mu) : U' = N' \cup M' \dss \to  U = N \cup M,$$
    which sends $x' \in N'$ to $\lm(x')$ and $y' \in M'$ to
    $\mu(y')$,
    is a tangible transmission (cf. Definition~\ref{defn2.3}).

    \item In this way we obtain all tangible transmissions from
    $U'$  to $U$.
\end{enumerate}
\end{prop}
\begin{proof} (i): A straightforward check.

(ii): Obvious.
\end{proof}

We mention in passing, that given an \m-valuation $v: R \to M$
with support $\iv(0) = \mfq$, the supertropical semiring $U^0(v)$
occurring in Theorem \ref{thm6.4} may be viewed as an instance of
Construction \ref{constr7.2}, as follows.

\begin{example}\label{exm7.4} Let $E$ denote the equivalence
relation on $R$ with equivalence classes $[0]_E = \mfq$ and $[x]_E
= \{ x \} $ for $x \in R\sm \mfq$. It is multiplicative, hence
gives us a monoid $R / E$ with absorbing element $[0]_E = 0$,
which we identify with the subset $(R \sm \mfq) \cup \{ 0 \}$ in
the obvious way.  The map $v: R\to M$ induces a monoid
homomorphism $\brv: R/E \to M$ with values $\brv(x) = v(x)$ for $x
\in R \sm \mfq$, $\brv(0) = 0 $. We have $\brv^{-1} (0) = \{ 0 \}$
and
$$U^0(v) = \STR(R / E, M , \brv).$$ \endbox
\end{example}

We now look for ways to ``unfold'' an arbitrary supertropical
monoid $U$. By this we roughly mean a fiber contraction $\tau:
\tlU \to U$ with $\tlU$ an unfolded supertropical monoid and
fibers $\itau(x)$, $x\in U$ as small as possible. More precisely
we decree

\begin{defn}\label{defn7.4}
Let $M := eU$, and let $N$ be a submonoid of $(U, \cdot \;)$ which
contains the set $\tT(U)_0$. An \textbf{unfolding of $U$ along
$N$} is a fiber contraction $\tau: \tlU \to U$ over $M$ (in
particular $e \tlU = M$), such that
$$  \itau(x)  = \left\{
\begin{array}{lll}
  \{x, \tlx \} & \text{if} &  x \in M \cap N, \\[1mm]
  \{ \tlx \} & \text{if} &  x \in N \sm M, \\[1mm]
    \{x \} & \text{if} &  x \in M \sm N, \\
\end{array}
\right.
$$
with $\tlx \in \tT(U)_0$. For any $x \in N$  we call $\tlx$ the
\textbf{tangible lift} of $x$ (with respect to $N$).
\end{defn}

Notice that this forces $\tau(\tT(\tlU)_0)= N$, and that moreover
for any $x \in N$ the tangible fiber~$\tlx$ is the unique element
of $\tT(\tlU)_0$ with $\tau(\tlx) = x$, hence $\tT(\tlU)_0 = \{
\tlx \ds | x\in N \}$.

Thus, if $\tau:\tlU \to U$ is an unfolding along  $N$, then the
map $\tlx \mapsto x$ from $\tT(\tlU)_0$ to $N$ obtained from
$\tau$ by restriction is a monoid isomorphism, and $\tau$ itself
is an ideal compression  with ghost kernel
 ${(M \cap N)}^{\sim} \cup M$, where ${(M
\cap N) }^{\sim}:= \{ \tlx \ds | x \in M \cap N \}$.

\begin{thm}\label{thm7.5} $ $
\begin{enumerate} \eroman
    \item Given a pair $(U,N)$ consisting of a supertropical
    monoid $U$ and a multiplicative submonoid $N \supset
    \tT(U)_0$, there exists an unfolding $\tau: \tlU \to U$ of $U$
    along $N$.

    \item  If $\tau': \tlU' \to U$ is a second unfolding of $U$
    along $N$, then there exists a unique isomorphism of
    supertropical monoids $\al : \tlU \ds \iso \tlU'$ with $\tau' \circ \al =
    \tau$.
\end{enumerate}

\end{thm}

\begin{proof} i) \emph{Existence:} Since $M$ is an ideal of $U$,
the set $M \cap N$ is a monoid ideal of $N$. We have $U = N \cup
M$, since  $N \supset
    \tT(U)_0$.  Let $\rho: N \to M$ denote the restriction of
    $\nu_U$ to $N$. It is a monoid homomorphism with $\irho(0) = \{ 0
    \}$.

    Let $\tlN$ denote a copy of the monoid $N$ with copying
    isomorphism $x \mapsto \tlx$ ($x \in N$), and let $\tlrho: \tlN \to
    M$ denote the monoid homomorphism from $\tlN$ to $M$
    corresponding to $\rho: N \to M$. Thus $\tlrho (\tlx) = \rho(x) =
    ex$ for $x\in N$. Now define the unfolded supertropical monoid
    $$ \tlU := \STR(\tlN, M, \tlrho) = \tlN \cup M.$$
In $\tlU$ we have $\tl0_U = 0 $ and $\tlN \cap M = \{ 0 \}$.
Further $\tT(\tlU)_0 = \tlN$ and $e \tlU = eU = M$.

We obtain a well-defined surjective map $\tau: \tlU \to U$ by
putting $\tau(\tlx) := x$ for $x\in N$, $\tau(y) := y$ for $y\in
M$. This map $\tau$ is multiplicative, as checked easily, sends
$0$ to $0$, $1 \in \tT(\tlU)$ to $1 \in N$, and restricts to the
identity on $M$. Thus $\tau$ is a fiber contraction (cf.
Definition~\ref{defn2.1}). The fibers of $\tau $ are as indicated
in Definition~\ref{defn7.4}; hence $\tau$ is an unfolding of $U$
along $N$. \pSkip

ii) \emph{Uniqueness}: Let $\tltau: \tlU \to U$ and $\tltau':
\tlU' \to U$ be unfoldings of  $U $ along $N$ with tangible lifts
$x \mapsto \tlx$ and $x \mapsto \tlx'$ respectively. Without loss
of generality we assume that $\tlU  = \STR(\tlN,M, \tlrho)$ and
$\tlU' = \STR(\tlN',M, \tlrho')$ with tangible lifts $x \mapsto
\tlx$ and $x \mapsto \tlx'$ ($x \in N$). Then $\tlrho(\tlx) =
\tlrho'(\tlx') = ex$ for every $x \in N$.  The map $\lm: \tlN \to
\tlN'$, given by $\lm (\tlx) = \tlx'$ for $x \in N$, is a monoid
isomorphism with $\tlrho' \circ \lm = \id_M \circ \tlrho$. Thus we
have a well defined transmission
$$ \al := \STR(\lm, \id_M): \tlU \ds \to \tlU'$$ at hand (cf. Proposition
\ref{prop7.3}). $\al$ is an isomorphism over $U$, i.e., an
isomorphism with $\tau' \circ \al = \tau$, clearly the only one.
\end{proof}

\begin{notation}\label{notation7.6}
We call the map $\tau: \tlU \to U$ constructed in part i) of the
proof of Theorem~\ref{thm7.5} \textbf{``the'' unfolding of $U$
along $N$} and write this map more precisely as
$$ \tau_{U,N}: \tlU(N) \to U$$ if necessary.   But sometimes we
abusively will denote any unfolding of $U$ along $N$ in this way,
justified by part ii) of Theorem \ref{thm7.5}.
\end{notation}

\begin{example}\label{exam7.7} We consider the very special case
that $U = eU =M$. Then $N$ can be any submonoid of $M$ containing
$0$. Now $\tlM(N) = \tlN \cup M$ with $\tlN \cap M = \{0\}$, and
$$\tlM(N) \cong \STR(N,M, i)$$ with $i : N \into M$ the inclusion
mapping. For every $x \in N$ there exists a unique tangible
element $\tlx$ of $\tlM(N)$ with $e \tlx =x$, while for $x \in M
\sm N$ there exists no such element.
\end{example}

\begin{thm}\label{thm7.8} Assume that $\al : U' \to U$ is a
transmission between supertropical monoids, and that $N' \supset
\tT(U')_0$, $N \supset \tT(U)_0$ are submonoids of $U'$ and $U$
with $\al(N') \subset N$. Then there exists  a unique tangible
transmission
$$ \tlal:= \tlal_{N',N} : \tlU' (N') \ds \to \tlU(N),$$
called the \textbf{tangible unfolding of $\al$ along $N'$ and
$N$}, such that the diagram
$$
\xymatrix{   \tlU'(N')   \ar @{>}[d]^{\tau_{U',N'}}  \ar
@{>}[rr]^{\tlal } &   & \tlU(N)
\ar @{>}[d]^{\tau_{U,N}}   \\
U'  \ar @{>}[rr]_{\al}  &  & U  }
$$
commutes.
\end{thm}
\begin{proof} Let $M' := eU'$, $M := eU$, and let $\rho': N' \to
M$, $\rho: N \to M$ denote the monoid homomorphism obtained from
$\nu_{U'}$ and $\nu_U$ by restriction to $N'$ and $N$. Then
$$\tlU'(N') = \STR(N',M', \rho'), \qquad \tlU(N) = \STR(N,M,\rho). $$
The map $\al$ restricts to monoid homomorphisms $\lm: N' \ds \to
N$ and $\gm: M' \to M$ with $\lm(0) = 0$, $\gm(0) = 0 $, and $\gm$
order preserving. Now $\gm \circ \nu_{U'} = \nu_U \circ \al$,
hence $\gm \rho' = \rho \lm$. Thus we have the tangible
transmission
$$ \tlal:= \STR(\lm, \gm): \tlU'(N') \ds \to \tlU(N)$$
at hand. Clearly  $\tau_{U,N} \circ \tlal = \al \circ
\tau_{U',N'}$. Since any tangible transmission from $\tlU'(N')$ to
$\tlU(N)$ maps $\tlN'$ to $\tlN$ and $M'$ to $M$, it is evident
that $\tlal$ is the only such map.
\end{proof}

\begin{cor}\label{cor7.9} Assume that $\al: U' \to U$ is a
transmission between supertropical monoids which is tangibly
surjective, i.e., $\tT(U) \subset \al(\tT(U'))$. Assume further
that $U'$ is unfolded. Let $N:= \al (\tT(U')_0)$, which is a
submonoid of $U$ containing $\tT(U)_0$.
\begin{enumerate} \eroman
    \item There exists a unique tangible transmission
    $$\tlal: U' \ds \to \tlU(N), $$
    called the \textbf{tangible lift} of $\al$, such that $\tau_{U,N} \circ \tlal = \al.$
    \item If $x' \in U'$, then
    $$  \tlal(x')  = \left\{
\begin{array}{lll}
  \widetilde{\al(x')} & \text{if} &  x' \in \tT(U')_0, \\[1mm]
  \al(x') & \text{if} &  x' \in eU'. \\
\end{array}
\right.
$$
\end{enumerate}
\end{cor}

\begin{proof} (i): applying Theorem~\ref{thm7.8} with $N' :=
\tT(U')_0$, and observe that $\tlU'(N') = U'$, since ~$U'$ is
unfolded. \pSkip

(ii): Now obvious, since $\tau_{U,N} (\tlal(x')) = \al (x')$ and
$\tlal(x') \in \tT(\tlU)_0$ iff $x' \in \tT(U')_0$.
\end{proof}
\pSkip  We are   ready to construct ``tangible lifts'' of
\m-supervaluations.

\begin{thm}\label{thm7.10}
Assume that $\vrp: R \to U$ is an \m-supervaluation which is
tangibly surjective, i.e., $\tT(U) \subset \vrp(R)$ \{e.g. $\vrp$
is surjective; $U = \vrp(R) \cup e \vrp(R)$\}. Let $N := \vrp(R)$,
which is a submonoid of $U$ containing $\tT(U)$.
\begin{enumerate} \eroman
    \item The map
    $$ \tlvrp : R \to \tlU(N), \qquad a \mapsto \widetilde{\vrp(a)},$$
    with $\widetilde{\vrp(a)}$ denoting the tangible lift of
    $\vrp(a)$ w.r.t. $N$, is a tangible \m-supervaluation of $\vrp$,
    called the \textbf{tangible lift} of $\vrp$.

    \item If $\vrp' : R \to U'$ is a tangible \m-supervaluation
    dominating $\vrp$, then $\vrp'$ dominates $\tlvrp$.
\end{enumerate}
\end{thm}
\begin{proof} (i): $\tlvrp$ is multiplicative, $\tlvrp(0) = 0$,  $\tlvrp(1) =
1$, and $e \tlvrp = e \vrp$ is an \m-valuation. Thus $\tlvrp$ is
an \m-supervaluation. By construction $\tlvrp$ is tangible. \pSkip

(ii): We may assume that the \m-supervaluation $\vrp': R \to U'$
is surjective, and hence   $\vrp'(R) \supset \tT'(U)_0$. Since
$\vrp'$ is tangible, this forces $\vrp'(R) =\tT'(U)_0$. Thus
$\tT'(U)_0$ is a submonoid of $U'$, i.e., $U'$ is unfolded. Since
$\vrp'$ dominates $\vrp$, there exists a transmission $\al: U' \to
U$ with $\vrp = \al \circ \vrp'$. We have  $$\al(\tT'(U)_0) = \al
(\vrp'(R)) = \vrp (R) = N.$$ Thus we have the  tangible lift
$$\tlal: U' \ds \to \tlU(N)$$ of $\al$ at hand. For any $a\in R,$
$$ \tlal(\vrp'(a)) = [\al (\vrp'(a))]^\sim  = \widetilde{\vrp(a)} = \tlvrp(a).$$
Thus $\tlvrp = \tlal \circ \vrp',$ which proves that $\vrp'$
dominates $\tlvrp.$
\end{proof}

\begin{addendum}\label{adden7.11} As the proof  has shown, if the
\m-valuation $\vrp'$ is surjective, then $U'$ is unfolded, and the
transmission $$\al_{\tlvrp,\vrp'}: U' \ds \to \tlU(N)$$ (cf.
\cite[Definition~5.3]{IKR1}) is the  tangible lift of $\al_{\vrp,
\vrp'}: U' \to U$.
\end{addendum}

\begin{cor}\label{cor7.12} If $\vrp$, $\psi$ are
\m-supervaluations covering $v$ and $\vrp \leq \psi$, then $\tlvrp
\leq \tlpsi$.
\end{cor}
\begin{proof} We have  $\vrp \leq \psi \leq \tlpsi$. It follows by
Theorem~\ref{thm7.10}.ii that $\tlvrp \leq \tlpsi$.
\end{proof}

\section{The partial tangible lifts of an
$\m$-supervaluation}\label{sec:8}

In all the following $v: R \to M$ is a fixed \m-valuation and
$\vrp: R \to U$ is a tangible surjective \m-supervaluation
covering $v$. (Most often $v$ and $\vrp$ will both be surjective.)
In \S\ref{sec:7} we introduced the tangible  lift $\tlvrp: R \to
\tlU$ (cf. Theorem 7.10). We now strive for an explicit
description of the \m-supervaluations $\psi$ covering $v$ with
$\vrp \leq \psi \leq \tlvrp$.

We warm up with two general observations.

\begin{defn}\label{defn8.1}
If $\psi: R \to V$ is an \m-supervaluation covering $v: R \to M$,
we call
$$ G(\psi) := \psi(R) \cap M = \{ \psi(a) \ds | a\in R, \ \psi(a)= v(a) \}$$
the \textbf{ghost value set} of $\psi$. \{Notice that $eV = M$.\}
\end{defn}

\begin{lem}\label{lem8.2}
Let $\psi_1, \psi_2$ be \m-supervaluations covering $v$. If
$\psi_1 \geq \psi_2$, then $G(\psi_1) \subset G(\psi_2)$. If
$\psi_1 \sim \psi_2$, then $G(\psi_1) = G( \psi_2)$.
\end{lem}

\begin{proof} Let $a\in R$. If $\psi_1 \geq \psi_2$, then $\psi_1(a) \in
M$ implies that $\psi_2(a) \in M$, due to condition $\D 3$ in the
definition of dominance (cf. Definition \ref{defn6.2}). Thus, for
$\psi_1 \sim \psi_2$ we have  $\psi_1(a) \in M$ iff $\psi_2(a) \in
M$.
\end{proof}

\begin{lem}\label{lem8.3} Assume that the \m-valuation $v: R \to
M$ is surjective. Then the ghost value set $G(\psi)$ of any
\m-supervaluation $\psi$ covering $v$ is an ideal of the semiring
$M$.
\end{lem}

\begin{proof}
If $x \in G(\psi)$ and $y \in M$, there exist $a,b \in R$ with
$\psi(a) =x$, $e \psi(b) = y$. It follows that
$$ xy = e \psi(a) \psi(b) = \psi(a) \psi(b) = \psi(a b).$$
Thus $xy \in \psi(R) \cap M = G(\psi)$. This proves that $G(\psi)
\cdot M \subset G(\psi)$. Since $M$ is bipotent, $G(\psi)$ is also
closed under addition.
\end{proof}

\begin{thm}\label{thm8.4} Assume that $\vrp: R \to U$ is an
\m-supervaluation  covering $v: R \to M$, and that $\psi_1,
\psi_2$ are \m-supervaluations covering $v$ with $$ \vrp \leq
\psi_1 \leq \tlvrp, \quad  \vrp \leq \psi_2 \leq \tlvrp.$$

\begin{enumerate} \eroman
    \item $\psi_1 \geq \psi_2 \dss \Leftrightarrow G(\psi_1) \subset
    G(\psi_2)$. \pSkip
    \item $\psi_1 \sim \psi_2 \dss \Leftrightarrow G(\psi_1) =
    G(\psi_2)$.
\end{enumerate}
\end{thm}

\begin{proof} We assume without loss of generality that $\vrp$ is
surjective. Then also the \m-super-valuations $\psi_1,$ $\psi_2,$
$ \tlvrp$ are surjective. By Corollary \ref{cor7.12} the tangible
lifts $\tlpsi_1$ and $\tlpsi_2$ are both equivalent to~$\tlvrp$.

Again without loss of generality we moreover assume that $\vrp =
\tlvrp / E := \pi_E \circ \tlvrp$ with $E$ an \MFCE-relation on
$\tlU$, and also $\psi_i = \tlvrp / E_i$ with an \MFCE-relation
$E_i$ ($i =1,2$). Let us describe these relations $E, E_1, E_2$
explicitly. We have $\tlU = \tlN \cup M$, with
$$ \tlN := \tT(\tlU)_0 = \tlvrp(R), \qquad \tlN \cap M = \{ 0\},$$
further $U = N \cup M$ with
$$ N:= \vrp (R), \qquad N\cap M = G(\vrp),$$
and we have a copying isomorphism $$s: N \ds \iso \tlN$$ of
monoids (new notation!), which sends each $x \in N$ to its
tangible lift $\tlx$, as explained in \S\ref{sec:7} (Definition
\ref{defn7.4}, Proof of Theorem \ref{thm7.5}.i). Notice that
$es(x)=x $ for $x \in N \cap M = G(\vrp).$

The relation $E$ has the 2-point equivalence classes $\{x, s(x)
\}$ with $x$ running through $G(\vrp)$, while all other
$E$-equivalence classes are one-point sets. Analogously, $E_i$ has
the 2-point set equivalence classes $\{x, s(x) \}$ with $x$
running through $G(\psi_i) \subset G(\vrp)$, while again all other
$E$-equivalence classes are one-point sets. Thus it is obvious
that\footnote{As in \cite{IKR1} we view an equivalence relation on
a set $X$ as a subset of $X \times X$ in the usual way.} $E_1
\subset E_2$ iff $G(\psi_1) \subset G(\psi_2)$. But $E_1 \subset
E_2$ means that $\psi_1 \geq \psi_2$. This gives  claim (i), and
claim (ii) follows.
\end{proof}

\begin{defn}\label{defn8.5} We call the monoid isomorphism
$$ s: \vrp(R) \ds \to \tT(\tlU)_0 = \tlvrp(R),$$
i.e., the copying isomorphism $s: N \iso \tlN$ occurring in the
proof of Theorem \ref{thm8.4}, the \textbf{tangible lifting map
for} $\vrp$.
\end{defn}

Notice that for $x\in \vrp(R)$, $y\in \tT(\tlU)_0$ we have $s(x)y
= s(xy)$. \pSkip

\emph{We assume henceforth that the \m-valuation $v: R \to M$ is
surjective, and that $\vrp: R \to U$ is a surjective
\m-supervaluation with $e \vrp = v$.} The question arises whether
every ideal $\mfa$ of~$M$ with $\mfa \subset G(\vrp)$ occurs as
the ghost value set $G(\psi)$ of some \m-supervaluation $\psi$
covering $v$ with $\vrp \leq \psi \leq \tlvrp$. This is indeed
true.

\begin{construction}\label{const8.6}
We employ the tangible lifting map $ s: \vrp(R) \ds \to \tlvrp(R)
= \tT(\tlU)_0$ defined above. Assume that $\mfa$ is an ideal of
$M$ contained in $G(\vrp)$. We have $$s(\mfa) \cdot \tlU \ds
\subset s(\mfa) \cup M,$$ since $s(x)y = s(xy) \in s(\mfa)$ for $x
\in \mfa$ and $y \in \tT(U)$. We conclude that $s(\mfa) \cup M$ is
an ideal of $U$. Let
$$ E_\mfa : = E(\tlU, s(\mfa)) = E(\tlU, s(\mfa)\cup M)$$
and $\tlU_\mfa := \tlU/ E_\mfa$. We regard $\tlU_\mfa$  as a
subset of $\tlU$, as indicated in Convention \ref{convs3.2}.a. The
map $\pi_{E_\mfa} : \tlU \onto \tlU_\mfa$ is the ideal compression
with ghost kernel $s(\mfa) \cup M$, and
$$ \tlvrp_\mfa  := \tlvrp/ E_\mfa = \pi_{E_\mfa} \circ \tlvrp : R \dss \to \tlU_\mfa$$
is an \m-supervaluation. For any $a\in R$
    $$  \tlvrp_\mfa (a)  = \left\{
\begin{array}{lll}
 \vrp(a) = v(a) & \text{if} \ \vrp(a) \in \mfa , &\\[1mm]
  \tlvrp(a) & \text{else}   \ . &  \\
\end{array}
 \right.
$$
Clearly $\vrp \leq \tlvrp_\mfa \leq \tlvrp$ and $G(\tlvrp_\mfa) =
\mfa$. We call $\tlvrp_\mfa$ \textbf{the tangible lift of $\vrp$
outside $\mfa$}, and we call any such map $\tlvrp_\mfa$ a
\textbf{partial tangible lift of $\vrp$.}
\end{construction}

Let $[\vrp, \tlvrp]$ denote the ``interval'' of the poset $
\Cov_m(v)$ containing all classes $[\psi]$ with $\vrp \leq \psi
\leq \tlvrp$, and let $[0,G(\vrp)]$ be the set of  ideals $\mfa$
of $M$ with $\mfa \subset G(\vrp)$, ordered by inclusion. By the
Lemmas  \ref{lem8.2} and \ref{lem8.3} we have a well defined order
preserving map
$$ [\vrp, \tlvrp] \ds \to [0, G(\vrp)],$$ which sends each class
$[\psi] \in [\vrp, \tlvrp]$ to the ideal $G(\psi)$. By Theorem
\ref{thm8.4} this map is injective, and by Construction
\ref{const8.6} we know that it is also surjective. Thus we we have
proved

\begin{thm}\label{thm8.7}
The map
$$ [\vrp, \tlvrp] \ds \to [0, G(\vrp)], \qquad [\psi] \mapsto G(\psi),$$
is a well defined order preserving bijection. The inverse of this
map sends an ideal $\mfa \subset G(\vrp)$  to the class
$[\tlvrp_\mfa]$ of the tangible lift of $\vrp$ outside $\mfa$.
\end{thm}

The  poset $\Cov_m(v)$ is a complete lattice (cf. Corollary
\ref{cor6.5}). The poset $I(M)$ consisting of the ideals of $M$
and ordered by inclusion,  is a complete lattice as well.  Indeed,
the infimum of a family $(\mfa_i \ds | i\in I)$ in $I(M)$ is the
ideal $\bigcap_i \mfa_i$,   while the supremum is the ideal
$\sum_i \mfa_i = \bigcup_i \mfa_i.$    \{Recall once more that
every subset of $M$ is closed under addition.\} The intervals
$[\vrp, \tlvrp]$ and $[0,G(\vrp)]$ are again complete lattices,
and thus the map $ [\vrp, \tlvrp] \ds \to [0, G(\vrp)]$ in Theorem
\ref{thm8.7} is an anti-isomorphism of complete lattices. This
implies the following

\begin{cor}\label{cor8.8}
Assume that $(\psi_i \ds | i \in I)$ is a family  of
supervaluations covering $v$ with $\vrp \leq \psi_i \leq \tlvrp$
for each $i \in I$. Let $\bigvee_i \psi_i$ and $\bigwedge_i
\psi_i$ denote respectively representatives of the classes
$\bigvee_i [\psi_i]$ and $\bigwedge_i [\psi_i]$ (as described in
\cite[\S7]{IKR1}). Then
$$G\bigg(\bigvee_i \psi_i \bigg) = \bigcap_i G(\psi_i),  \qquad  G\bigg(\bigwedge_i \psi_i\bigg) = \bigcup_iG(\psi_i).$$
\end{cor} \pSkip

We switch to the case where $\vrp: R \to U$ is a supervaluation,
i.e., the supertropical monoid~ $U$ is a semiring. We want to
characterize the partial tangible lifts $\psi$ of $\vrp$ which are
again supervaluations; in other terms, we want to determine the
subset $[\vrp, \tlvrp] \cap \Cov (v)$ of the interval $[\vrp,
\tlvrp]$ of $\Cov_m(v)$.

The set $Y(v)$ introduced near the end of \S\ref{sec:6} will play
a decisive role. It consists of the products $a b \in R$ of
elements $a,b \in R$  for which there exists some $a' \in R$ with
$$v(a') < v(a), \quad v(a'b) = v(ab) \neq 0.$$
Henceforth we call these products $ab$ the
\textbf{$v$-NC-products} (in $R$). Let  further
$$ \mfq' := \mfq \cup Y(v)$$
with $\mfq$ the support of $v$, $\mfq = v^{-1}(0)$. As observed in
\S\ref{sec:6}, $\mfq'$ is an ideal of the monoid $(R, \cdot \;)$,
while $\mfq$ is an ideal of the semiring $R$.
\begin{examp}\label{exmp8.9}

Let $R$ be a supertropical semiring and $\gm: eR \to M$ a semiring
homomorphism to a bipotent semiring $M$. Then $$v:= \gm \circ
\nu_R: R \ds \to M$$ is a strict \m-valuation. The $v$-NC-products
are the products $yz$ with $y,z \in U$ such that there exists some
$y' \in R$ with $$\gm(ey') < \gm (ey), \quad \gm(ey'z) = \gm
(eyz).$$ Thus $Y(v)$ is the ideal $D_0(R,\gm)$ of the
supertropical semiring $R$ introduced in Definition~\ref{defn4.8}.

\end{examp}
\begin{prop}\label{prop8.10} If $\vrp$ is a supervaluation then
$\vrp(\mfq')$ is contained in the ghost value set~ $G(\vrp)$.
\end{prop}

\begin{proof} We have seen in \S\ref{sec:6} that $\vrp(Y(v)) \subset M$. Since $\vrp(\mfq) = \{
0\}$, this implies that $\vrp(\mfq') \subset M \cap \vrp(R) =
G(\vrp)$.
\end{proof}

\begin{rem}\label{rem8.11} Here is a more direct argument that $\vrp(Y(v))
\subset M$, than given in the proof of Theorem \ref{thm6.12}.i. If
$x \in Y(v)$, then we have $a',a,b \in R$ with $x =ab$, $v(a') <
v(a)$, $v(a'b) = v(ab) \neq 0$. Clearly $\vrp(x) = \vrp (a) \vrp
(b)$ is an NC-product in the supertropical semiring~ $U$ (recall
Definition \ref{defn4.2}), and thus $\vrp(x)$ is ghost, as
observed already in Theorem \ref{thm1.2}.
\end{rem}

\begin{lem}\label{lem8.12}
Assume that $\vrp: R \to U$ is a surjective tangible
\m-supervaluation covering $v$. Then $\vrp(R \sm \mfq) = \tT(U)$,
$v(R) = M$, and $\vrp(Y(v)) = S(U)$.\footnote{Recall that $S(U)$
denotes the set of tangible NC-products in $U$
(Definition~\ref{defn4.2}).}
\end{lem}

\begin{proof} a) We have $U = \vrp(R) \cup v(R)$, $\vrp(R \sm \mfq) \subset
\tT(U)$, and $v(R) \subset M$. Since $U = \tT(U) \dot \cup M$,
this forces $\vrp(R \sm \mfq) = \tT(U)$ and $v(R) = M.$ \pSkip

b) Let $c \in Y(v)$. There exist $a,b,a' \in R$ with $c =ab$,
$v(a') < v(a)$, $v(a'b) = v(ab) \neq 0.$ It follows that $\vrp(c)
= xy \neq 0$ with $x:= \vrp(a)$, $y:= \vrp(b)$, $v(a') < ex$,
$v(a') y = exy.$ Thus $\vrp(c)$ is an NC-product in $U$. Moreover
$\vrp(c)$ is tangible, hence $\vrp(c) \in S(U).$  Thus $\vrp(Y(v))
\subset S(U).$ \pSkip

c) Let $x \in S(U)$  be given. Then $x = yz \in \tT(U)$ with $y,z
\in U$ and $y' < ey$, $y'z = eyz \neq 0$ for some $y' \in M$.
Clearly $y,z\in \tT(U)$. We choose $a, b, a' \in R$ with $\vrp(a)
= y$, $\vrp(b) = z$, $v(a') = y'.$ Then $ey = v(a)$, $ez = v(b)$,
and it follows that $v(a') < v(a)$, $v(a'b) = v(ab) \neq 0$. Thus
$ab \in Y(v)$ and $x = \vrp(ab)$. This proves that $S(U) \subset
\vrp (Y(v))$.
\end{proof}

\begin{thm}\label{thm8.13} Assume that $\vrp: R \to U$ is a
supervaluation, i.e., $U$ is a semiring. Let $\chvrp$ denote the
tangible lift of $\vrp$ outside the ideal $v(\mfq') = \{ 0 \} \cup
v(Y(v))$ of $M$,
$$ \chvrp := (\tlvrp)_{v(\mfq')} : R \ds \to \chU := \tlU/E_{v(\mfq')}$$
(cf. Construction \ref{const8.6}).
\begin{enumerate} \eroman
    \item $\chvrp$ is again a supervaluation. More precisely,
    $\chvrp$ coincides with the supervaluation $(\tlvrp)^\wedge$
    associated to the tangible lift $\tlvrp: R \to \tlU$ of $\vrp$
    (cf. Definition \ref{defn6.7}). \pSkip

    \item If $\psi$ is an \m-supervaluation covering $v$ with $\vrp \leq \psi \leq
    \tlvrp$, then $\psi$ is a supervaluation iff $\psi \leq
    \chvrp$. Thus
    $$ [\vrp, \tlvrp] \cap \Cov(v) = [\vrp, \chvrp].$$
\end{enumerate}
\end{thm}

\begin{proof} (i): $(\tlvrp)^\wedge$ is the map $\tlvrp/ E(\tlU, S(\tlU)) = \pi_{E(\tlU, S(\tlU)} \circ
\tlvrp$ from $R$ to $\chU:= \tlU/ E(\tlU, S(\tlU).$ By Lemma
\ref{lem8.12}, applied to $\vrp$, we have
$$ S(\tlU) \cup \{ 0 \} = \tlvrp (\mfq') = s\vrp(\mfq')$$
with $s: \vrp(R) \to \tT(\tlU)_0$ denoting the tangible lifting
map for $\vrp$. Moreover $\vrp(\mfq') = v(\mfq')$ by Proposition
\ref{prop8.10}. Thus $ \chU = \tlU/ E_{v(\mfq')}$ and
$(\tlvrp)^\wedge = \tlvrp/ E_{v(\mfq')} = \chvrp.$ \pSkip

(ii): If $\psi$ is a supervaluation, then we know by Proposition
\ref{prop8.10} that $G(\psi) \supset v(\mfq') = G(\chvrp)$, and
hence by Theorem \ref{thm8.4} that $\psi \leq \chvrp$. Conversely,
if $\psi \leq \chvrp$, then $\psi$ is a supervaluation since
$\chvrp$ is a supervaluation (cf. Proposition \ref{prop6.6}).
\end{proof}
\begin{defn}\label{defn8.14}$ $ \begin{enumerate} \eroman
    \item Given a supervaluation $\vrp:= R \to U$ covering $v$ we
    call
    $$ \chvrp := (\tlvrp)^\wedge: R \ds \to \chU = (\tlU)^\wedge$$
the \textbf{almost tangible lift of $\vrp$} (to a supervaluation)
and we call $[\chvrp] \in \Cov(v)$ the almost tangible  lift (in
$\Cov(v)$) of the class $[\vrp] \in Cov(v)$. \pSkip

    \item If $\chvrp = \vrp$, we say that $\vrp$  itself  is \textbf{almost tangible.}
\end{enumerate}

\end{defn}

\begin{rems}\label{rem8.15}
$ $ \begin{enumerate} \ealph
    \item Clearly $\vrp$ is almost tangible iff $G(\vrp) =
    v(\mfq')$. A subtle point here is that then there may
    nevertheless exist elements $a \in R \sm \mfq'$  with
    $\vrp(a)$ ghost.

    \item If $\vrp$ is any supervaluation, then
    $\chvrp$ is almost tangible.

    \item If $v$ happens to be a valuation, i.e., $M$ is
    cancellative, then $\chvrp = \tlvrp$.
\end{enumerate}
\end{rems}

\begin{prop}\label{prop8.16}
If $\psi$ is any almost tangible supervaluation dominating the
supervaluation~$\vrp$ (but not necessarily covering $v$), then
$\psi$ dominates $\chvrp$.
\end{prop}

\begin{proof} $\tlpsi \geq \tlvrp$, and hence $\psi= (\tlpsi)^\wedge \geq (\tlvrp)^\wedge = \chvrp.$
\end{proof}

\end{document}